\documentclass[11pt]{amsart}
\usepackage[english]{babel}
\usepackage{a4wide, amsmath,graphicx}
\usepackage{amssymb}
\usepackage{amsfonts, amsthm}
\usepackage{todonotes}

\usepackage[parfill]{parskip}

\usepackage{color}
\usepackage{xcolor}

\usepackage{enumerate}

\usepackage[colorlinks=true, pdfstartview=FitV,linkcolor=cyan,citecolor=cyan, urlcolor=black]{hyperref}
\usepackage{cleveref}
\crefformat{equation}{(#2#1#3)} 
\makeatletter
\def\subsection{\@startsection{subsection}{2}%
  \z@{.5\linespacing\@plus.7\linespacing}
{.5\baselineskip}%
  {\normalfont\centering\scshape}%
}
\makeatother

\numberwithin{equation}{section}

\def\R{\mathbb{R}}

\def\N{\mathbb{N}}

\def\r3{\R^3}

\newcommand{\ben}{\begin{eqnarray*}}
\newcommand{\een}{\end{eqnarray*}}

\newtheorem{theorem}{Theorem}
\newtheorem{proposition}{Proposition}
\newtheorem*{proposition*}{Proposition}

\newtheorem{lemma}{Lemma}
\newtheorem{remark}{Remark}

\begin{document}

\title[Contractivity of Wasserstein distance and other functionals]{
  Contractivity of Wasserstein distance and other functionals for the Landau equation
  }

\author{F.-U. Caja-Lopez, M. G. Delgadino, M.-P. Gualdani, M. Taskovic}

\address{Department of Mathematics, The University of Texas at Austin (USA)}
\email{funai.caja@utexas.edu}
\email{matias.delgadino@utexas.edu}
\email{gualdani@math.utexas.edu}
\address{Department of Mathematics, Emory University (USA)}
\email{maja.taskovic@emory.edu}

\begin{abstract}
  Following the breakthrough work of Guillen and Silvestre \cite{GS24}, that shows that the Fisher information is monotonically decreasing for solutions to the homogeneous Landau equation, 
  we study, for the same equation,  the monotonicity properties of other physically relevant functionals. In the case of Maxwellian molecules, we show that the relative $L^2$ norm with respect to the equilibrium decays exponentially fast in time and is monotonically decreasing after some time. Moreover, still for the Maxwellian case, we provide a novel and short quantitative proof of time monotonicity of the  entropic Wasserstein metric. 
  
  For soft potentials, we show that the Wasserstein metric is contractive, conditional to $L^1(0,T,L^p(\mathbb{R}^3))$ bound for the solution. This result provides an alternative proof of the   Fournier and  Fournier-Guerin uniqueness theorem in  \cite{fournier2009well_posedness_soft_potentials}  \cite{fournier2010uniqueness_Coulomb} . 
  
\end{abstract}

\keywords{Landau equation, Maxwellian  molecules, entropic Wasserstein distance, relative $L^2$-norm,  exponential decay, contractivity.}

\maketitle


\section{Introduction}
In this paper, we consider the space-homogeneous Landau equation
\begin{equation}\label{L}
  \partial_t f = \textnormal{div}_{v}\int_{\mathbb{R}^{d}}
  \alpha(\vert v-w\vert )\vert v-w\vert^2 K(v-w)\left(f(w)\nabla_{v}f(v)-f(v)\nabla_{w}f(w)\right)\,dw,
\end{equation}
with $  K(z) := \textnormal{Id} - \frac{z\otimes z}{\vert z \vert^2}$ being the orthogonal projection to $\langle z\rangle^\perp$ and $\alpha:\mathbb{R}_+\longrightarrow\mathbb{R}_+$. 
The most relevant physical case is $d=3$ and $\alpha(r)=r^{-3}$, which corresponds to Coulomb interactions of electrons, and arises in mathematical modeling of electrons in plasma.  The Landau equation is one of the most used mathematical models in plasma physics. From a mathematical point of view, the presence of  nonlinear and non-local terms makes the analysis quite challenging. After decades of being an open problem, well posedness of \eqref{L} with smooth data has  been recently settled in the breakthrough work of Guillen-Silvestre \cite{GS24}. The proof relies on a novel a-priori estimate which shows that the Fisher information is non-increasing in time. Subsequently, global well-posedness for less regular initial data was shown in \cite{GoGuLo24, DGL24, Ji24, HJJ2025}.  The surprising monotonicity of the Fisher information hints at the possibility of additional functionals that either decay monotonically or stay bounded. 

In a preliminary attempt to answer this question, we focus here on three specific functionals. The first one is the relative $L^2$ distance between any solution to (\ref{L}) and the Maxwellian equilibrium. The other two are the classical $2$-Wasserstein distance and the  entropic Wasserstein distance between two solutions to (\ref{L}). We summarize our main results in three theorems.

The first theorem shows that in the case of Maxwell molecules  ($\alpha\equiv 1$), the relative $L^2$ distance decays exponentially fast in time.  Moreover, this distance is monotonically decreasing  after a short initial time. This transient period {\em{measures}} how far the initial data is from a radially symmetric configuration. If the initial data is {\em{almost}} radially symmetric, the relative $L^2$ norm is monotone decreasing for all times. 

\begin{theorem}\label{thm:rel_L2_monotone}
  Let $f$ be a solution to (\ref{L})  with $\alpha \equiv 1$ and initial data $f_{in}$ such that 
  $ \int f_{in} (1, v, |v|^2) \;dv = (1,0,d)$.   Let $m$ be the corresponding unique equilibrium solution, $m(v) :=(2\pi)^{-\frac{d}{2}}e^{-\frac{1}{2}\vert{v}\vert^{2}}$. Then the relative $L^2$-norm 
  \begin{equation*}
 E(t):=   \int_{\R^d}\left(\frac{f}{ m }-1\right)^{2} m \;dv
  \end{equation*}
 decays exponentially fast  in time, 
 $$
 E(t) \le c_0 e^{-2(d-1)t}E(0),
 $$
 where $c_0$ depends only on the dimension $d$. Moreover, $E(t)$ is monotonically decreasing for any $t \ge  t_0$ with $t_0 
 = \max\left\{\frac{1}{4d}\log\left(\frac{2\sup_i (\lambda_i^{in}-1)}{d-1}\right),0\right\}$, where $\lambda_i^{in}$ are the eigenvalues of the matrix $\mathbb{P}^{in}_{ij}=\int v_iv_jf_{in}\,dv$.
 \end{theorem}
{\bf{Remark.}}
If $f_{in}$ satisfies $\lambda_i^{in}\leq \frac{d+1}{2}$ for all $i = 1,...,d$,  then $t_0=0$, namely the relative $L^2$ norm decreases from the initial time. In particular this applies to radially symmetric initial data, which have ${\lambda_i^{in}=1}$.

Our second result shows that, still for the Maxwell molecules case, the Landau equation is contractive with respect to the entropic $2$-Wasserstein distance. Before we define the distance, let's recall the definition of Boltzmann entropy $\mathcal H(f)$ and relative entropy $\mathcal H(f|g)$ between two probability measures $f$ and $g$ as 
$$
\mathcal H(f) := \int f\log f \;dx,
\qquad
\mathcal H(f|g) := \int \left( \frac{f}{g}\log \frac{f}{g} -\frac{f}{g} +1\right) g\;dx .
$$

The entropic $2$-Wasserstein distance is defined as
\begin{equation}\label{eq:defn_W2}
  \mathcal T_\varepsilon(\mu,\nu) = \inf_{\Pi\in \Gamma(\mu,\nu)} 
  \left\{ \iint|v-w|^2 \, d\Pi(v,w) + \varepsilon\mathcal H(\Pi|\mu\otimes \nu)\right\},
\end{equation}
where $\mu,\nu$ are two probability measures over $\mathbb{R}^d$, and $\Gamma(\mu,\nu)$ stands for the set of transport plans, which are probability measures over $\mathbb{R}^d\times \mathbb{R}^d$ with marginals $\mu$ and $\nu$, see \cite{gigli2020benamou} for further details. Before we state our second main theorem, we recall a useful expression introduced in \cite{GS24}  for the matrix $\vert v-w\vert^2 K(v-w)$: 
$$
\vert v-w\vert^2 K(v-w) = \sum_{k=1}^3 b_k(v-w) \otimes  b_k(v-w),
$$
with 
  \begin{equation*}
    b_{1}(v-w):=\begin{pmatrix}0\\
      w_{3}-v_{3}\\
      v_{2}-w_{2}
      \end{pmatrix},\; b_{2}(v-w):=\begin{pmatrix}v_{3}-w_{3}\\
      0\\
      w_{1}-v_{1}
      \end{pmatrix},\; b_{3}(v-w):=\begin{pmatrix}w_{2}-v_{2}\\
      v_{1}-w_{1}\\
      0
      \end{pmatrix}.
  \end{equation*}
In the following of the manuscript we denote by $\tilde b_k$ the vector 
$$
      \tilde b_k := \begin{pmatrix}
        b_k \\ -b_k
      \end{pmatrix} . 
      $$ %

Our second main result is summarized in the following theorem:

\begin{theorem}\label{thm:w2_dissipation}
  Let $f$, $g$ be solutions to (\ref{L}) for $\alpha \equiv 1$ in $\mathbb{R}^3$. Then $\mathcal T_\varepsilon (f,g)$ is non-increasing for any $\varepsilon\geq 0$. Moreover, the dissipation can be expressed as 
  \begin{equation}\label{eq:w2_dissipation2}
    \frac{d}{dt}\mathcal T_\varepsilon(f,g) = \varepsilon \underbrace{\frac{d}{dt} \mathcal H(f)}_{\le 0} - \frac{1}{2}\sum_k\int_0^1\iint_{\mathbb{R}^3\times\mathbb{R}^3} \left|\nabla(\tilde b_k\cdot \nabla U)\right|^2 P \,dv\,dw\,ds\le 0,
  \end{equation}
  where $P:[0,1]\to\mathcal{P}(\mathbb{R}^3\times\mathbb{R}^3)$ is the Wasserstein geodesic between $F(v,w)=f(v)f(w)$ and $G(v,w)=g(v)g(w)$ with vector field given by ${\nabla U:[0,1]\to L^2(\mathbb{R}^3\times\mathbb{R}^3,dP)^{\otimes 6}}$. 
\end{theorem}

The Geodesic $P$ solves a continuity equation with $\nabla U$ as the driving vector field, and $U$ solves a Hamilton-Jacobi-Bellman equation. See Section \ref{sec:w2_method2} for the precise statements and explicit expression of $U$.

We provide two different proofs of the monotonicity of $\mathcal T_\varepsilon$. The first one closely follows the symmetrization argument introduced in  \cite{GS24} (see Section \ref{sec:monotonicity_W2}). The second proof relies on the duality between entropic optimal transport and Hamilton-Jacobi-Bellman equations; thanks to the duality argument, we are able to find the explicit expression of the dissipation (\ref{eq:w2_dissipation2}). This approach is developed in Section \ref{sec:w2_method2}


We remark that the contractivity of the 2-Wassertein distance $\mathcal T$ for Maxwellian molecules was already known \cite{Fournier_Guillin,Carrapatoso}. Even more classically, we mention the work of Tanaka \cite{Tanaka}, which first showed the monotonicity of the 2-Wasserstein metric for the Boltzmann equation. We also mention \cite{CarrilloDelgadinoDesvillettesWu2024} which produces a tailored 2-Wasserstein like metric on which the Landau equation arises as the gradient flow of the entropy;  understanding if the Landau equation is contractive in this natural metric is an interesting open problem. Here we provide a shorter and more intuitive proof of the monotonicity of $\mathcal T$. The expression of the dissipation on the right hand side of (\ref{eq:w2_dissipation2}) is at the authors' best  knowledge new.


The last main result concerns soft potentials and the contractivity of the $2$-Wasserstein  $ \mathcal T(f_{t},g_{t})$ distance. Recall that $ \mathcal T(f_{t},g_{t})$ is defined in (\ref{eq:defn_W2}) with $\varepsilon =0$.

%
\begin{theorem}\label{thm:w2_estimate_non_Coulomb}
  Let $f_t$, $g_t$ be two smooth solutions to (\ref{L}) in $\mathbb R^3$ with $\alpha(r)=r^{\gamma}$. 
  
  $(i)$   Let $-3<\gamma <0$. Assume that $f_t,g_t\in L_t^{1}L_{v}^{p}$ for $p>\frac{3}{3+\gamma}$. 
  
 There exist constants  $C_{\gamma,p}>0$ and $0<\beta_{\gamma,p}<1$ such that
  \begin{equation}\label{eq:contrac_estimate_non_coulomb}
    \mathcal T(f_{t},g_{t})\leq \mathcal T(f_{in},g_{in})\exp\left(C_{\gamma,p}\int_{0}^{t}\left(\left\Vert f_{s}\right\Vert _{L^{p}(\mathbb{R}^{3})}^{\beta_{\gamma,p}}+\left\Vert g_{s}\right\Vert _{L^{p}(\mathbb{R}^{3})}^{\beta_{\gamma,p}}\;ds\right)\right) .
  \end{equation}

 $(i)$ Let  $\gamma=-3$. Assume $f_t,\; g_t\in L_{t}^{1}L_{v}^{\infty}$, and  initial data such that  $\mathcal{T}(f_{in},g_{in})<e^{-4}$. Then there exist constants $C>0$ and $\tau_0>0$ such that for $t<\tau_0$ : 
  \begin{equation}\label{eq:contract_estimate_coulomb}
    \mathcal{T}(f_{t},g_{t})\leq\mathcal{T}(f_{in},g_{in})^{e^{-\beta(t)}},\quad\beta(t):=C\int_{0}^{t}\max\big(\left\Vert f_{s}\right\Vert _{L^{\infty}(\mathbb{R}^{3})}+\left\Vert g_{s}\right\Vert _{L^{\infty}(\mathbb{R}^{3})},1\big)\,ds.
  \end{equation}
\end{theorem}

The contractive estimates for soft potentials of Theorem \ref{thm:w2_estimate_non_Coulomb}  match the  estimates of Fournier and Guerin  \cite{fournier2009well_posedness_soft_potentials} and Fournier \cite{fournier2010uniqueness_Coulomb}. Indeed, our result shows that uniqueness of solution for (\ref{L}) for Coulomb potential holds provided the solution belongs to $L^1(0,\tau,L^\infty(\mathbb{R}^3))$ for small $\tau>0$. Our proof does not rely on stochastic methods, but solely uses the symmetrization properties of the distance $\mathcal{T}$ and of the Landau operator.  

It turns out that Theorem \ref{thm:w2_estimate_non_Coulomb} holds for a wider range of potentials. More precisely, it holds for any general  potential  $\alpha(r) $ that satisfies 
$$
\alpha'(r)^2\leq Kr^{\gamma-2}\alpha(r),  
$$
for some positive constant $K$ and some $\gamma \in [-3,0]$.

\noindent \textbf{Organization of the paper.}  The rest of the paper is organized as follows.  Section \ref{sec:collision_operator_MM} collects some important properties  of the Landau operator with Maxwellian molecules. Section \ref{sec:rel_L2} contains the proof of Theorem \ref{thm:rel_L2_monotone}.  Theorem \ref{thm:w2_dissipation} is proven using convexity properties in Section \ref{sec:monotonicity_W2}   and via duality arguments in  Section \ref{sec:w2_method2}.  Finally, Section \ref{sec:proof_w2_estimate_non_Coulomb} contains the proof of Theorem~\ref{thm:w2_estimate_non_Coulomb}.

\section{The collision operator for Maxwellian molecules}\label{sec:collision_operator_MM}

The proof of Theorem \ref{thm:rel_L2_monotone} ($\alpha \equiv 1$) starts with the Landau equation (\ref{L})  rewritten in the following form:
%
\begin{equation}\label{eq:Landau_MM_linear}
  \partial_t f = \textnormal{div}\left( 
      \left[ \left\vert v \right\vert ^{2}K(v) + d\, \textnormal{Id}  - \mathbb{P}(t) \right]\nabla f
      + (d-1)vf
     \right),
\end{equation}
where $\mathbb P(t)$ is the diagonal matrix formed by
\begin{equation}\label{eq:lambda_i}
  \lambda_i(t) = \int v_i^2f(v)\,dv,\qquad \lambda_i(t) = 1 + e^{-4dt}(\lambda_i(0)-1).
\end{equation}
%

%

We briefly show how to derive the expression (\ref{eq:Landau_MM_linear}) assuming 
$$\int f_{\textnormal{in}}(1,v,|v|^2)\,dv =\int f(1,v,|v|^2)\,dv =(1,0,d).$$ This has already been done in \cite{desvillettes_villani2000landau_hard}, but we present a slightly modified proof for completeness, as these computations will be used in the proof of Theorem \ref{thm:rel_L2_monotone}.
\begin{proposition}
Let $f\ge 0$ such that $\int f(1,v,|v|^2)\,dv =(1,0,d)$.     The matrix 
  $$A[f](v,t):=\int_{\mathbb{R}^{d}}{K(v-w)}{\vert{v-w}\vert^{2}}f(w,t)\,dw$$
  in the Maxwell molecules case $ \alpha \equiv 1$ can be rewritten as 
  $$A[f](v,t)=\left\vert v \right\vert ^{2}K(v)+d\,\textnormal{Id} - \mathbb{P}(t)$$
   where $\mathbb{P}(t)$ is the matrix with entries $P_{ij}(t):=\int w_iw_jf(w,t)\,dw$. 
   
   Moreover, the Landau equation (\ref{L}) reduces to
  \begin{equation}\label{eq:Landau_max_mol}
    \partial_t f = \textnormal{div}\left( 
      \left[ \left\vert v \right\vert ^{2}K(v) + d\, \textnormal{Id}  - \mathbb{P}(t) \right]\nabla f
      + (d-1)vf
     \right).
  \end{equation}
\end{proposition}
\begin{proof}
Direct computations show that
  \begin{equation*}
    A[f]=\textnormal{Id}\int\left\vert v-w \right\vert ^{2}f(w)\,dw-\int(v-w)\otimes(v-w)f(w)\,dw.
  \end{equation*}
  For the first term we have
  \begin{equation*}
    \int\left\vert v-w \right\vert ^{2}f(w)\,dw=\Bigg[\left\vert v \right\vert ^{2}\int f(w)\,dw-2v\cdot\int wf(w)\,dw+\int\left\vert w \right\vert ^{2}f(w)\,dw\Bigg]=\left(\left\vert v \right\vert ^{2}+d\right).
  \end{equation*}
  For the second:
  \begin{align*}
    \int(v-w)\otimes(v-w)f(w)\,dw &= v\otimes v\int f(w)\,dw-2v\otimes\int f(w)w\,dw + \int\left(w\otimes w\right)f(w)\,dw \\
     &= (v\otimes v)+ \int\left(w\otimes w\right)f(w)\,dw = (v\otimes v)+\mathbb{P}(t).
  \end{align*}
 Hence we have
  \begin{equation*}
    A[f] = \left\vert v \right\vert ^{2}\Bigg(\textnormal{Id}-\frac{v\otimes v}{\left\vert v \right\vert ^{2}}\Bigg)+d\,\textnormal{Id} - \mathbb{P}(t) = \left\vert v \right\vert ^{2}K(v)+d\,\textnormal{Id} - \mathbb{P}(t).
  \end{equation*}
  This also provides a very simple expression for
  \begin{equation*}
    \textnormal{div}\left(A[f]\right)=\textnormal{div}\left(\left\vert v \right\vert ^{2}\textnormal{Id}-v\otimes v\right)
    =\left[2v-(d+1)v\right]=-(d-1)v.
  \end{equation*}
  Since (\ref{L}) can be expressed in terms of $A[f]$ as
  \begin{equation*}
    \partial_{t}f=\textnormal{div}\left(A[f]\nabla f-f\,\textnormal{div}A[f]\right),
  \end{equation*}
  substituting $ A[f] $ and $\textnormal{div}\left(A[f]\right)$ above we obtain (\ref{eq:Landau_max_mol}). 
\end{proof}
%
%
{\bf{Remark.}} If  the solution is rotationally symmetric $f(t,v)=f(t,\left\vert v \right\vert ^{2})$, equation (\ref{eq:Landau_max_mol}) reduces to the classical Fokker Planck equation. Let $F$ be an antiderivative of $f$, for each $t$ and $i\neq j$ we have
\begin{equation*}
  P_{ij}(t)=\int w_{i}w_{j}f(\left\vert w \right\vert ^{2})\,dw=\int w_{i}\partial_{w_{j}}\left(\frac{F(\left\vert w \right\vert ^{2})}{2}\right)\,dw=0.
\end{equation*}
For the case $i=j$, due to radial symmetry, the area under $w_{i}^{2}f(\left\vert w \right\vert ^{2})$ is the same for every $i$, so 
\begin{equation*}
  P_{ii}(t)=\int w_{i}^{2}f(\left\vert w \right\vert ^{2})\,dw=\frac{1}{d}\int\left\vert w \right\vert ^{2}f(\left\vert w \right\vert ^{2})\,dw=1.
\end{equation*}
This shows that $\mathbb{P}(t)=\textnormal{Id}$. Moreover, we have 
\begin{equation*}
  A[f]\nabla f(v)  =\left(\left\vert v \right\vert ^{2}K(v)+(d-1)\textnormal{Id}\right)2vf'(\left\vert v \right\vert ^{2})
  =(d-1)2vf'(\left\vert v \right\vert ^{2})=(d-1)\nabla f(v)
\end{equation*}
and the Landau equation for Maxwell molecules reduces to
\begin{equation*}
  \partial_{t}f=(d-1)\textnormal{div}\left(\nabla f(v)+vf(v)\right).
\end{equation*}
Next, we show how $A[f]$ evolves in time:

\begin{proposition}\label{thm:ODE_coeff}
 Let $f$ be a solution to (\ref{eq:Landau_max_mol}) with  $\int f(1,v,|v|^2)\,dv =(1,0,d)$. The matrix $\mathbb{P}$ with entries $P_{ij}(t)=\int w_i w_j f(w)\,dw$ satisfies the ode
  \begin{align} \label{ode_lambda_i}
  \mathbb{P}'(t) = -4d (\mathbb{P}(t)-\textnormal{Id}),
  \end{align}
 with explicit solution
 $$\mathbb{P}(t)= \textnormal{Id} + e^{-4dt}(\mathbb{P}_{in}-\textnormal{Id}).$$
\end{proposition}
\begin{proof}
  Using the equation for $f$, and $\textnormal{div}\left(A[f]\right)=-(d-1)v$ we deduce
  \begin{align*}
    \frac{d}{dt}\int w_{i}w_{j}f(w)\,dw & =\int w_{i}w_{j}\textnormal{div}_{v}\bigg(A[f]\nabla f\,dw-\textnormal{div}_{v}\left(A[f]\right)f\bigg)\,dw\\
     & =-\int\left\langle \nabla(w_{i}w_{j}),A[f]\nabla f \right\rangle\,dw +\int\left\langle \nabla(w_{i}w_{j}),\textnormal{div}\left(A[f]\right) \right\rangle f\,dw\\
     & =\int\textnormal{div}\left(A[f]\nabla(w_{i}w_{j})\right)f\,dw-(d-1)\int\left\langle \nabla(w_{i}w_{j}),w \right\rangle f\,dw.
  \end{align*}
  The first term simplifies as
  \begin{align*}
    \textnormal{div}\left(A[f]\nabla(w_{i}w_{j})\right) & =\textnormal{div}\left(A[f]\right)\cdot\nabla(w_{i}w_{j})+\textnormal{tr}\left(A[f] D^{2}(w_{i}w_{j})\right)\\
      & =-(d-1)w\cdot\nabla(w_{i}w_{j})+\textnormal{tr}\left(A[f] D^{2}(w_{i}w_{j})\right).
  \end{align*}
  Note that $D^{2}(w_{i}w_{j})=\Delta_{ij}+\Delta_{ji}$, where $\Delta_{ij}$ is the matrix with all zeroes except a one at position $(i,j)$. In general, if $A$ is a matrix then $A\Delta_{ij}$ has all of its columns equal to $0$ except the $j$-th, where you find the $i$-th column of $A$. Therefore, if the matrix $A$ is symmetric, we have $\textnormal{tr}\left(AD^{2}(w_{i}w_{j})\right)=2A_{ij}$, which shows:
  \begin{equation*}
    \textnormal{div}\left(A[f]\nabla(w_{i}w_{j})\right)=
    -(d-1)w\cdot\left(\vec{e}_{i}w_{j}+\vec{e}_{j}w_{i}\right)+2A_{ij}[f]
    =-2(d-1)w_{i}w_j+2A_{ij}[f].
  \end{equation*}
  Similarly, we have
  \begin{equation*}
    \int\left\langle \nabla(w_{i}w_{j}),w \right\rangle f\,dw=2\int w_{i}w_{j}f(w)\,dw,
  \end{equation*}
  which shows 
  \begin{equation}\label{eq:evol_A_Landau_homo_maxwel}
    \frac{d}{dt}\int w_{i}w_{j}f(w)\,dw=2\int A_{ij}[f]f(w)\,dw-4(d-1)\int w_{i}w_{j}f(w)\,dw.
  \end{equation}
  Now we use the following expression for $A[f]$
  \begin{align*}
    A[f] &	=\left\vert v\right\vert ^{2}K(v)+d\,\textnormal{Id}- \mathbb{P}(t) \\
    & =\left\vert v\right\vert ^{2}K(v)
    +\int\vert w\vert^{2}\left(\textnormal{Id}-\frac{w\otimes w}{\vert w\vert^{2}}\right)f(w)\,dw,
  \end{align*}
  which we plug into (\ref{eq:evol_A_Landau_homo_maxwel}), yielding
  \begin{align*}
    2\int A_{ij}[f](w)f(w)\,dw & =2\int\Bigg(\left\vert w \right\vert ^{2}K(w)+\int\left\vert \xi \right\vert ^{2}K(\xi)f(\xi)\,d\xi\Bigg)f(w)\,dw\\
      & =4d\,\textnormal{Id}-4\mathbb{P}(t),
  \end{align*}
  which finishes the proof.
\end{proof}
{\bf{Remark.}} The previous formula allows one to simplify the expression of $\mathbb{P}(t)$. Let $f_{in}$ be a general initial data with second moment matrix $ \mathbb{P}_{in}$. Let $R$ be the  rotation matrix such that $R\, \mathbb{P}_{in} R^{t}=\textnormal{diag}\left(\lambda_{1}^{\textnormal{in}},\ldots,\lambda_{d}^{\textnormal{in}}\right)$. Then,  the rotated initial data $\tilde{f}_{in}(v):=f_{in}\left(R^{-1}v\right)$ has diagonal second moment matrix: 
\begin{equation*}
  \mathbb{P}_{in}=\int(v\otimes v)f_{in}\left(R^{-1}v\right)\,dv=\int(Rv)\otimes(Rv)f_{in}(v)\,dv=R\,\mathbb{P}(0)R^{t}
  =\textnormal{diag}\left(\lambda_{1}^{\textnormal{in}},\ldots,\lambda_{d}^{\textnormal{in}}\right).
\end{equation*}
Without loss of generality, in Section \ref{sec:rel_L2} we will consider initial data such that $ \int w_i w_j f_{in} \;dw =0$ if $ i \neq j$. Then $\mathbb{P}_{in}$ is a diagonal matrix and, in view of (\ref{ode_lambda_i}), $\mathbb{P}(t)$ is also diagonal with elements $\lambda_{i}(t)=1+e^{-4dt}\left(\lambda_{i}^{\textnormal{in}}-1\right)$. Because of conservation of the second moment we have that  for all times $t\ge 0 $:
$$
\textnormal{tr}( \mathbb{P}(t)) = \sum_{i =1}^d \lambda_i(t) = d.
$$

\section{Exponential decay of relative $L^2$ norm}\label{sec:rel_L2}
We start the proof Theorem \ref{thm:rel_L2_monotone} with the following computation: 
\begin{lemma}\label{lem:ineq_lambda_minus1}
  Let $f$ be a solution of (\ref{eq:Landau_max_mol}) with initial data such that $\mathbb{P}_{in}$ is diagonal. We have
  \begin{align*}
    \frac{1}{2}\frac{d}{dt}\int\left(\frac{f}{ m }-1\right)^{2} m \,dv \leq 
   & -(d-1)\int \left\vert\nabla\left(\frac{f}{ m } \right)\right\vert^{2} m \,dv \\
   & +\sum_{i}(\lambda_{i}(t) -1)\int\left(\left[\partial_{v_{i}}\left(\frac{f}{ m } \right)\right]^{2}-\frac{1}{2}\frac{f^{2}}{ m ^{2}}v_{i}^{2}\right) m\,dv,
  \end{align*}
  where $\lambda_{i}(t)=\int w_i^2 f(w)\,dw=1+e^{-4dt}\left(\lambda_{i}^{\textnormal{in}}-1\right)$.
\end{lemma}
\begin{proof}
 We rewrite (\ref{eq:Landau_max_mol}) as 
 \begin{align*}
  \partial_{t}f  =\textnormal{div}\left(\vert v\vert^{2}K(v)\nabla f\right)+d\,\textnormal{div}\left( m \nabla \left(\frac{f}{ m } \right)\right)-\textnormal{div}\left(fv+\mathbb{P}\nabla f\right).
\end{align*}
Because of conservation of mass,  $\int f^2/m\,dv$ and $\int \left(f/m-1\right)^2m\,dv$ only differ by a constant, so
  \begin{align*}
    \frac{1}{2}\frac{d}{dt}\int\left(\frac{f}{ m }-1\right)^{2} m\,dv & 
    = \int \frac{f}{m}\partial_tf\,dv \\
    & = -\int\left\langle K(v)\left\vert v \right\vert ^{2}\nabla f(v),\nabla\left(f/ m \right)\,dv \right\rangle \\ 
    & -d\int\left\vert \nabla\left(f/ m \right) \right\vert ^{2} m\,dv 
    +\int\left\langle fv+\mathbb{P}\nabla f,\nabla(f/ m ) \right\rangle\,dv .
  \end{align*}
  Since $K(v) \nabla m =0$, the first term in the last expression is negative:
  \begin{equation*}
    \left\langle K(v)\left\vert v \right\vert ^{2}\nabla f(v),\nabla\left(f/ m \right) \right\rangle =\left\langle \left\vert v \right\vert ^{2}K(v)\nabla(f/ m ),\nabla\left(f/ m \right) \right\rangle  m \geq0.
  \end{equation*}
  
  Therefore, we have
  \begin{equation*}
    \frac{1}{2}\frac{d}{dt}\int\left(\frac{f}{ m }-1\right)^{2} m\,dv \leq
    -d\int\left\vert \nabla\left(f/ m \right) \right\vert ^{2} m\,dv + I_1 + I_2,
  \end{equation*}
  where $I_1 := \int fv\cdot\nabla(f/ m )\,dv$ and $I_2 :=\int\left\langle \mathbb{P}\nabla f,\nabla(f/ m ) \right\rangle\,dv$. Expanding $\nabla (f/m)$ in $I_1$ yields
  \begin{align*}
    I_1 =\int fv\cdot\frac{\nabla f+fv}{ m }\,dv & =\int\frac{\left\vert v \right\vert ^{2}}{ m }f^{2}\,dv+\int\frac{v}{ m }\cdot\overbrace{f\nabla f}^{\frac{1}{2}\nabla(f^{2})}\,dv\\
     & =\int\frac{\left\vert v \right\vert ^{2}}{ m }f^{2}\,dv-\frac{1}{2}\int f^{2}\left(\frac{d+\left\vert v \right\vert ^{2}}{ m }\right)\,dv \\
     & =\frac{1}{2}\int\frac{f^{2}}{ m }\left(\left\vert v \right\vert ^{2}-d\right)\,dv.
  \end{align*}
  The identity $\nabla f= m \nabla(f/ m )-fv$ yields 
  \begin{align*}
    I_2 &	=\int\left\langle  \mathbb{P} \left( m \nabla(f/ m )-fv\right),\nabla(f/ m ) \right\rangle \,dv\\
    & = \int\left\langle  \mathbb{P} \nabla(f/ m ),\nabla(f/ m ) \right\rangle  m \,dv -\int\left\langle  \mathbb{P} v,f\nabla(f/ m ) \right\rangle\,dv .
  \end{align*}
  We rewrite the last term as 
  \begin{align*}
    -\int\left\langle  \mathbb{P} v,f\nabla(f/ m ) \right\rangle\,dv  & =-\frac{1}{2}\int \mathbb{P} v\cdot\nabla\left(f^{2}/ m ^{2}\right)\, m\,dv \\
     & =\frac{1}{2}\int\frac{f^{2}}{ m ^{2}}\left( m \,\textnormal{div}\left( \mathbb{P} v\right)- \mathbb{P} v\cdot v m \right)\,dv\\
     & =\frac{d}{2}\int\frac{f^{2}}{ m }\,dv-\frac{1}{2}\int\frac{f^{2}}{ m }\left\langle  \mathbb{P} v,v \right\rangle\,dv ,
  \end{align*}
  since $\textnormal{tr}( \mathbb{P}(t)) = d$.  This yields
  \begin{align*}
    I_1 + I_2 & 
    = \frac{1}{2}\int\frac{f^{2}}{ m }\left(\left\vert v \right\vert ^{2}-d\right) \,dv +
    \int\left\langle  \mathbb{P} \nabla(f/ m ),\nabla(f/ m ) \right\rangle  m\,dv \\
    &+ \frac{d}{2}\int\frac{f^{2}}{ m } \,dv -\frac{1}{2}\int\frac{f^{2}}{ m }\left\langle  \mathbb{P} v,v \right\rangle\,dv \\
    & = \frac{1}{2}\int\frac{f^{2}}{ m }\left\vert v \right\vert^{2}\,dv - \frac{1}{2}\sum_i \lambda_i \int \frac{f^2}{m}v_i^2\,dv +
    \sum_i \lambda_i \int \left[\partial_{v_{i}}(f/ m )\right]^{2}m\,dv\\
    & = -\frac{1}{2}\sum_i (\lambda_i - 1)\int \frac{f^2}{m}v_i^2\,dv + 
    \sum_i (\lambda_i - 1) \int \left[\partial_{v_{i}}(f/ m )\right]^{2}m\,dv + \int\vert\nabla\left(f/ m \right)\vert^{2} m\,dv,
  \end{align*}
  which proves the lemma.
\end{proof}
In the next result we show a gaussian weighted Poincare inequality, which is a (slightly) sharper version of the one found in \cite{Dolb_Volzone}.
\begin{lemma}\label{lem:improved_gauss_poincare}
  Let $m(x)=(2\pi)^{-\frac{d}{2}}e^{-\frac{|x|^2}{2}}$ be the Gaussian weight on $\mathbb{R}^d$. Then for any smooth function $\varphi:\mathbb{R}^d\to\mathbb{R}$ we have
  \begin{align*}
  \int_{\R^d} \vert x_i\vert^2\varphi^2(x) m(x)\,dx  \le {2} \int_{\R^d} \vert \partial_{x_i}(x) \varphi\vert^2m(x)\,dx + 2  \int_{\R^d} \varphi^2(x) m(x)\,dx,
  \end{align*}
  Moreover,  if $\varphi$ has zero average $\int_{\R^d} \varphi(x)m(x)\, dx = 0$ then we have
  \begin{equation*}
    \int_{\R^d} \varphi^2(x) m(x)\,dx   \leq \int \vert \nabla \varphi(x)\vert^2m(x)\,dx.
  \end{equation*}
\end{lemma}
\begin{proof}
  We consider the basis of Hermite polynomials $\{H^\alpha\}_{\alpha\in \N^d}$ defined by
  \begin{align*}
    H_\alpha(x) = \frac{(-1)^{|\alpha|}}{\sqrt{\alpha!}} e^{\frac{|x|^2}{2}} \partial_\alpha e^{-\frac{|x|^2}{2}}(x)
    = \textnormal{He}_{\alpha_1}(x_1) \cdots \textnormal{He}_{\alpha_d}(x_d),
  \end{align*}
  where $\alpha = (\alpha_1, \ldots, \alpha_d)$ is a multi-index and $\alpha! = \prod_{i=1}^d \alpha_i!$.
  The Hermite polynomials form a complete orthonormal basis with respect to the Gaussian measure $m(x)$. Any smooth function $\varphi$ can be expressed as
  \begin{align*}
    \varphi(x) = \sum_{\alpha} a_\alpha H_\alpha(x),
  \end{align*}
  where $a_\alpha = \int_{\R^d} \varphi(x) H_\alpha(x) m(x) \, dx$ are the coefficients of the expansion. We start with the first inequality. We have
  \begin{align*}
    \int_{\R^d} \vert x_1 \vert^2\varphi^2(x) m(x)\,dx &= \int_{\R^d}\left|\sum_{\alpha} a_\alpha x_1 H_\alpha \right|^2m(x)\,dx\\
    &=\int_{\R^d}\left|\sum_{\alpha} a_\alpha x_1\textnormal{He}_{\alpha_1}(x_1)\cdots \textnormal{He}_{\alpha_d}(x_d)\right|^2m(x)\,dx\\
    &=\int_{\R^d}\left|\sum_{\alpha} a_\alpha (\textnormal{He}_{\alpha_1+1}(x_1)+\textnormal{He}_{\alpha_1}'(x_1) )\cdots \textnormal{He}_{\alpha_d}(x_d)\right|^2m(x)\,dx,\\
  \end{align*}
  where we have used the identity $x\textnormal{He}_n=\textnormal{He}_{n+1}+\textnormal{He}_n'$. We thus have
  \begin{align*}
    \int_{\R^d} \vert x_1 \vert^2\varphi^2(x) m(x)\,dx&\le 2\int_{\R^d}\left|\sum_{\alpha} a_\alpha \textnormal{He}_{\alpha_1+1}(x_1)\cdots \textnormal{He}_{\alpha_d}(x_d)\right|^2m(x)\,dx\\
    &\qquad + 2\int_{\R^d}\left|\sum_{\alpha} a_\alpha \textnormal{He}_{\alpha_1}'(x_1)\cdots \textnormal{He}_{\alpha_d}(x_d)\right|^2m(x)\,dx\\
    &= 2\int_{\R^d} |\varphi(x)|^2 m(x)\,dx + 2\int_{\R^d} |\partial_{x_1}\varphi(x)|^2 m(x)\,dx.
  \end{align*}
  The Poincare inequality follows in a similar manner, using the identity $\textnormal{He}_n'=n\textnormal{He}_{n-1}$.
\end{proof}

We are now ready to prove Theorem \ref{thm:rel_L2_monotone}: 

\begin{proof}[Proof of Theorem \ref{thm:rel_L2_monotone}]
  We start with Proposition~\ref{lem:ineq_lambda_minus1}, using that $ \lambda_i-1 =e^{-4dt}(\lambda_i^\textnormal{in}-1)$:
  \begin{align*}
    \frac{1}{2}E'(t) & \leq
    -(d-1)\int\vert\nabla\left(f/ m \right)\vert^{2} m\,dv + 
    \sum_{i}(\lambda_{i}-1)\int\left(\left[\partial_{v_{i}}(f/ m )\right]^{2}-\frac{1}{2}\frac{f^{2}}{ m ^{2}}v_{i}^{2}\right) m\,dv\\
    & =  -(d-1)\int\vert\nabla\left(f/ m \right)\vert^{2} m\,dv + 
    e^{-4dt}\sum_{i}(\lambda_{i}^{in}-1)\int\left(\left[\partial_{v_{i}}(f/ m )\right]^{2}-\frac{1}{2}\frac{f^{2}}{ m ^{2}}v_{i}^{2}\right) m\,dv \\
     &=  \sum_{i=1}^d \left[ - \big(d-1-e^{-4dt}(\lambda_{i}^{in}-1)\big) \int \left[\partial_{v_{i}}(f/ m )\right]^{2}m\;dv  - \frac{1}{2}e^{-4dt}(\lambda_{i}^{in}-1) \int \frac{f^2}{m}v_i^2 \,dv\right].
  \end{align*}
 Because the $\lambda_i^{in}$ are nonnegative, and their sum equals to $d$, at most $d-1$ of them are greater than one and at most $d-1$ of them are less than one. Also, the maximum value of $(\lambda_i^{in}-1)$ is $d-1$ and the minimum is $-1$. Now, we note that
  \begin{align*}
    -\frac{1}{2}\sum_{i}(\lambda_{i}-1)\int\left(\frac{f}{m}-1\right)^{2}v_{i}^{2}m\,dv & =
    -\frac{1}{2}\sum_{i}(\lambda_{i}-1)\int\left(\frac{f^{2}}{m}v_{i}^{2}-2fv_{i}^{2}+mv_{i}^{2}\right)\,dv 	\\
    & =-\frac{1}{2}\sum_{i}(\lambda_{i}-1)\int\frac{f^{2}}{m}v_{i}^{2}\,dv+\sum_{i}(\lambda_{i}-1)\lambda_{i}, 
  \end{align*}
  since $\sum_i(\lambda_i-1)=0$. Moreover $\sum_{i}(\lambda_{i}-1)\lambda_{i} = \sum_{i}(\lambda_{i}-1)^2\geq 0$, which implies 
  $$
  -\frac{1}{2}\sum_{i}(\lambda_{i}-1)\int\frac{f^{2}}{m}v_{i}^{2}\,dv \le -\frac{1}{2}\sum_{i}(\lambda_{i}-1)\int\left(\frac{f}{m}-1\right)^{2}v_{i}^{2}m \,dv.
  $$
  Going back to the original estimate, we have that 
  \begin{align}
    \frac{1}{2}E'(t) & 
    \leq \sum_{i=1}^d\Bigg[ - \big(d-1-e^{-4dt}(\lambda_{i}^{in}-1)\big) \int \left[\partial_{v_{i}}(f/ m )\right]^{2}m\;dv\nonumber	\\
    & \qquad- \frac{1}{2}e^{-4dt}(\lambda_{i}^{in}-1) \int \left( \frac{f}{m}-1\right)^2v_i^2m \,dv\Bigg].\label{eq:ineq_lambda_minus1}
  \end{align}
    For the $\lambda^{in}_i $ greater or equal $1$, the quantity in the square bracket is nonpositive, and we just trivially bound 
\begin{align*}
    &- \big(d-1-e^{-4dt}(\lambda_{i}^{in}-1)\big) \int \left[\partial_{v_{i}}(f/ m )\right]^{2}m\;dv  - \frac{1}{2}e^{-4dt}(\lambda_{i}^{in}-1) \int \left(\frac{f}{m}-1\right)^{2}m v_i^2 \,dv\\
    &\qquad\le - \big(d-1-e^{-4dt}(\lambda_i^{\textnormal{in}}-1)\big) \int \left[\partial_{v_{i}}(f/ m )\right]^{2}m\;dv\le 0.
\end{align*}
  For the terms with  $\lambda^{in}_i<1$ in \eqref{eq:ineq_lambda_minus1} we use lemma \ref{lem:improved_gauss_poincare} (for $ \varphi = f/m-1$) to obtain
  \begin{align*}
    - \big(d-1-e^{-4dt}(\lambda_{i}^{in}-1)\big) \int \left[\partial_{v_{i}}(f/ m )\right]^{2}m\;dv  + \frac{1}{2}e^{-4dt}(1-\lambda_{i}^{in})\int \left( \frac{f}{m}-1\right)^2 mv_i^2\,dv\\
    \qquad\le - (d-1) \int \left[\partial_{v_{i}}(f/ m )\right]^{2}m\;dv+e^{-4dt}(1-\lambda_{i}^{in})\int \left( \frac{f}{m}-1\right)^2 m\,dv.
  \end{align*}

  Assume now the worst case scenario in which only one $\lambda_i^{\textnormal{in}}$ is greater or equal than $1$ and the rest are all less or equal than $1$. Without loss of generality let $\lambda_1^{in} \ge 1$ and $\lambda_j^{in}\le 1$ for $j\ge 2$.  We have that 
  \begin{equation*}
    \begin{array}{l}
      \displaystyle\frac{1}{2}\frac{d}{dt}\int \left(\frac{f}{ m }-1\right)^{2} m \,dv
  \le - (d-1-e^{-4dt}(\lambda_{1}^{in}-1)) \int \left[\partial_{v_{1}}(f/ m )\right]^{2}m\;dv \\
  \displaystyle \qquad-\sum_{i=2}^d (d-1)\int \left[\partial_{v_{i}}(f/ m )\right]^{2}m\;dv
    +  \sum_{i=2}^d {{{e^{-4dt}}(1-\lambda_{i}^{in})  \int  \left(\frac{f}{m}-1\right) ^{2} m \,dv}}\\
    \displaystyle \le  - (d-1-e^{-4dt}(\lambda_{1}^{in}-1))\int \left| \nabla(f/ m )\right|^{2}m\;dv+ e^{-4dt}(\lambda_{1}^{in}-1) \int \left(\frac{f}{m}-1\right)^{2}m\,dv   \\
    \displaystyle \le - {(d-1-2e^{-4dt}(\lambda_{1}^{in}-1))} \int \left(\frac{f}{m}-1\right)^{2}m\,dv ,
  \end{array}
  \end{equation*}
  where we have used the fact that $\sum_{i=2}^{d} (1-\lambda_{i}^{in}) =  (\lambda_{1}^{in}-1) $ and the Gaussian Poincare inequality. 
  The ODE
  $$
  y' \le - 2{(d-1-2e^{-4dt}(\lambda_{1}^{in}-1))}y
  $$
  has solution 
  $$
  y(t) \;\le\; y_0 \;e^{-2(d-1)t + \frac{(\lambda_{1}^{in}-1)}{d}( 1 - e^{-4dt})}
  $$
  and this finishes the proof. 
\end{proof}

One can also repeat the computations of Lemma \ref{lem:ineq_lambda_minus1} with the test function $\varphi = (f/m)^{p-1}$and obtain
\begin{eqnarray*}
  &&\frac{1}{p}\frac{d}{dt}\int(f/m)^{p}m\,dv\\
  &&\qquad\leq-\frac{4(p-1)}{p^{2}}\sum_{i}\left[(d-\lambda_{i})\int\left(\partial_{v_{i}}(f/m)^{\frac{p}{2}}\right)^{2}m\,dv+\frac{p}{4}(\lambda_{i}-1)\int(f/m)^{p}v_{i}^{2}m\,dv\right].
\end{eqnarray*}
As before, the quantity inside the bracket is nonpositive when $\lambda_{i}\geq1$.
For the case $\lambda_{i}<1$ we can use Lemma \ref{lem:improved_gauss_poincare} with $\varphi=(f/m)^{\frac{p}{2}}$, which shows
\[
\int(f/m)^{p}v_{i}^{2}m\,dv\leq 2\int\left(\partial_{v_{i}}(f/m)^{\frac{p}{2}}\right)^{2}m\,dv+2\int(f/m)^{p}m\,dv,
\]
yielding
\begin{align*}
  \frac{1}{p}\frac{d}{dt}\int(f/m)^{p}m\,dv &
  \leq-\frac{4(p-1)}{p^{2}}\sum_{\lambda_{i}<1}(d-\lambda_{i})\int\left(\partial_{v_{i}}(f/m)^{\frac{p}{2}}\right)^{2}m\,dv	\\
  & +\frac{2(p-1)}{p}\sum_{\lambda_{i}<1}(1-\lambda_{i})\int\left(\partial_{v_{i}}(f/m)^{\frac{p}{2}}\right)^{2}m\,dv\\
  &+\frac{2(p-1)}{p}\sum_{\lambda_{i}<1}(1-\lambda_{i})\int(f/m)^{p}m\,dv\\
  & \leq-\sum_{\lambda_{i}<1}\left(d-\lambda_{i}-\frac{p}{2}(1-\lambda_{i})\right)\int\left(\partial_{v_{i}}(f/m)^{\frac{p}{2}}\right)^{2}m\,dv\\
  & +\frac{2(p-1)(d-1)}{p}\sup_{i}(1-\lambda_{i})\int(f/m)^{p}m\,dv.
\end{align*}
%
Remembering that $\lambda_{i}(t)=1+e^{-4dt}\left(\lambda_{i}^{\textnormal{in}}-1\right)$, we have an inequality of the form
\[
E_{p}'(t)\leq c_{0}e^{-4dt}E_{p}(t),\quad c_{0}=2(p-1)(d-1)\sup_{i}(1-\lambda_{i}^{\textnormal{in}})
\]
for $t\geq t_{0}:=\frac{1}{4d}\log\left(\frac{p+2}{2(d-1)}\sup_i(1-\lambda_i)\right)$, which shows that 
\begin{equation*}
  E_{p}(t)\leq\tilde{c}_{0}E_{p}(t_{0}),\quad \tilde{c}_{0}=\exp\left(\frac{c_{0}}{4d}e^{-4dt_{0}}\right).
\end{equation*}
In summary, we have proven the following:
\begin{proposition}\label{prop:bounded_rel_Lp}
 Let $f$ be solution to (\ref{eq:Landau_max_mol}) with  $\int f_{in}(1,v,|v|^2)\,dv =(1,0,d)$.  Let $m$ be the corresponding unique equilibrium solution, $m(v) :=(2\pi)^{-\frac{d}{2}}e^{-\frac{1}{2}\vert{v}\vert^{2}}$ and consider the relative $L^p$ norm $E_p(t):=\int (f/m)^pm\,dv$. Then there exists a constant $C$ depending on $d$, $p$ and $\lambda_i^\textnormal{in}$ such that
  \begin{equation*}
    E_{p}(t)\leq C E_{p}(t_{0}),\quad t\geq t_{0}:=\frac{1}{4d}\log\left(\frac{p+2}{2(d-1)}\sup_i(1-\lambda_i)\right).
  \end{equation*}
\end{proposition}
\section{Contractivity of Entropic 2-Wasserstein distance}\label{sec:monotonicity_W2}

 In this section we show  that $\mathcal T_\varepsilon(f,g)$ is decreasing as time grows  for $f$ and $g$ any solutions of the Landau equation for Maxwell molecules.
To prove monotonicity of $\mathcal T_\varepsilon(f,g)$ we closely follow the strategy used in \cite{GS24} for the Fisher information. 
Let $q(f)$ be the Landau operator ($\alpha \equiv 1$)
%
\begin{equation}\label{eq:Landau_MM}
   q(f):=\textnormal{div}_{v}\int_{\mathbb{R}^{3}}\left|v-w\right|^{2}K(v-w)\left(\nabla_{v}-\nabla_{w}\right)\left[f(v)f(w)\right]\,dw,
\end{equation}
%
%
and define 
\begin{equation}\label{eq:defn_Q}
  Q(F):=\left(\textnormal{div}_{v}-\textnormal{div}_{w}\right)\Big[\left|v-w\right|^{2}K(v-w)\left(\nabla_{v}-\nabla_{w}\right)F\Big].
\end{equation}
When $F = f(v) f(w)$, the operator $Q(F)$ is connected with $q(f) $ through an integration:
\[
\int Q(f(v)f(w)) \;dw =q(f).
\]
In the following we will denote with $(f\otimes f)(v,w)$ the product 
$$
(f\otimes f)(v,w):=f(v)f(w).
$$


The next lemma states that the first variation of the Wasserstein distance between $f$ and $g$ along the Landau flow can be symmetrized, and the resulting expression is the first variation of the Wasserstein distance in double variables along a linear problem.

\begin{proposition}\label{prop:equality_W2_production}
  Consider the operators $q$, $Q$ given in (\ref{eq:Landau_MM}), (\ref{eq:defn_Q}) . Let  $F=f\otimes f$, $G=g\otimes g$. We have
  \begin{equation}\label{eq:dual_form_bHJB}
    \left\langle \mathcal{T}_{\varepsilon}'(f,g),\left(q(f),q(g)\right)\right\rangle =\frac{1}{2}\left\langle \mathcal{T}_{\varepsilon}'(F,G),\left(Q(F),Q(G)\right)\right\rangle.
  \end{equation}
\end{proposition}
\begin{proof}
  To prove (\ref{eq:dual_form_bHJB}) we use the following dual formulation (see for instance \cite{gigli2020benamou}):
  \begin{equation*}
    \mathcal{T}_{\varepsilon}(f,g)=\varepsilon \mathcal H(f)+ \mathcal J(f,g),\quad
    \mathcal J(f,g) := \sup\bigg\{ \int u_{1}g-\int u_{0}f:\underbrace{\partial_{s}u+\frac{1}{2}\left|\nabla u\right|^{2}+\frac{\varepsilon}{2}\Delta u=0}_{\text{bHJB equation}}\bigg\},
  \end{equation*}
  where $\mathcal H(f)$ is the entropy of $f$ and $u$ verifies the bHJB equation. As a first step, we show that the entropy functional verifies the desired equality. On the one hand, we have
  \begin{align*}
    \left\langle \mathcal{H}' \right. & (f), \left.q(f)\right\rangle  =\int\left(\log f(v)+1\right)q(f)\,dv=\int q(f) \log f(v)\,dv	\\
    &=-\iint\left\langle \left|v-w\right|^{2}K(v-w)\frac{\nabla f(v)}{f(v)},(\nabla_{v}-\nabla_{w})\left(f(v)f(w)\right)\right\rangle \,dv\,dw\\
    &=-\frac{1}{2}\iint\left\langle \left|v-w\right|^{2}K(v-w)\left(\frac{\nabla f(v)}{f(v)}-\frac{\nabla f(w)}{f(w)}\right),(\nabla_{v}-\nabla_{w})\left(f(v)f(w)\right)\right\rangle \,dv\,dw\\
    &=-\frac{1}{2}\iint\left|v-w\right|^{2}\left|K(v-w)\left(\frac{\nabla f(v)}{f(v)}-\frac{\nabla f(w)}{f(w)}\right)\right|^{2}f(v)f(w)\,dv\,dw.
  \end{align*}
  On the other hand, we have
  \begin{align*}
    \left\langle \mathcal{H}' \right.& (F)\left.,Q(F)\right\rangle  =\iint\left(\log F(v,w)+1\right)Q(F)(v,w)\,dv\,dw	\\
    &=\iint\left(\log f(v)+\log f(w)\right)Q(F)(v,w)\,dv\,dw \\
    & =-\iint\left\langle \frac{\nabla f(v)}{f(v)}-\frac{\nabla f(w)}{f(w)},\left|v-w\right|^{2}K(v-w)(\nabla_{v}-\nabla_{w})\left(f(v)f(w)\right)\right\rangle \,dv\,dw\\
    & =2\left\langle \mathcal{H}'(f),q(f)\right\rangle .
  \end{align*}
  Now we show that $\mathcal J(f,g)$ lifts in the same way as the entropy by relating the optimal $u$ to the optimal plan. It is known that the optimal plan between $f$ and $g$ is of the form
  \begin{equation*}
    \pi^*(v,w)=\exp\left\{\varepsilon^{-1}\tilde \phi_\varepsilon(v)+ \varepsilon^{-1}\tilde\psi_\varepsilon(w)\right\}r_{\frac{\varepsilon}{2}}(v,w),
  \end{equation*}
  where $r_t(v,w)$ denotes the heat kernel at the time $t$ and $\tilde \phi_\varepsilon$, $\tilde \psi_\varepsilon$ are the Schrodinger pontentials, which converge to the Kantorovich potentials $\tilde \phi$, $\tilde \psi$ for the cost $|v-w|^2$ as $\varepsilon\to 0$. Moreover, $\pi^*$ is the only plan $\pi\in\Pi(f,g)$ such that $\displaystyle\frac{d\pi}{dr_{\frac{\varepsilon}{2}}}$ is tensorized and the Kantorovich potentials provide the optimal $u$ for the dynamic formulation $\mathcal J(f,g)$, given by
  \begin{equation*}
    u=\varepsilon\log\left[h_{\frac{(1-s)\varepsilon}{2}}e^{\varepsilon^{-1}\tilde\psi_{\varepsilon}}\right],
  \end{equation*}
  where $h_t$ denotes the heat flow so that $\partial_t (h_tf) = \Delta (h_tf)$ and $h_0=\textnormal{Id}$. If we double the variables and consider $F=f\otimes f$, $G=g\otimes g$ then the optimal plan would be $\pi^* \otimes \pi^*$ and the optimal solution for $\mathcal J(F,G)$ would be $u(v)+u(w)$ giving:
  \begin{align*}
    \left\langle \mathcal{J}'(F,G),\left(Q(F),Q(G)\right)\right\rangle &
    =\iint\Big[\left(u_{1}(v)+u_{1}(w)\right)Q(G)-\left(u_{0}(v)+u_{0}(w)\right)Q(F)\Big]\,dv\,dw	\\
    &=2\iint u_{1}(v)Q(G)\,dv\,dw-2\iint u_{0}(v)Q(F)\,dv\,dw\\
    & =2\int u_{1}(v)\textnormal{div}_{v}\int\left|v-w\right|^{2}K(v-w)\left(\nabla_{v}-\nabla_{w}\right)G\,dw\,dv\\
    & -2\int u_{0}(v)\textnormal{div}_{v}\int\left|v-w\right|^{2}K(v-w)\left(\nabla_{v}-\nabla_{w}\right)F\,dw\,dv\\
    & =2\int u_{1}(v)q(g)\,dv-2\int u_{0}(v)q(f)\,dv\\
    & =2\left\langle \mathcal{J}'(f,g),\left(q(f),q(g)\right)\right\rangle,
  \end{align*}
  which finishes the proof.
\end{proof}

What the above proposition shows is that computing  $\frac{d}{dt} \mathcal T_\varepsilon(f,g)$ is equivalent to computing $\frac{d}{dt} \mathcal T_\varepsilon(F,G) _{|_{t =0}}$ where $F$ and $G$ solve the linear problem 
$$
\partial_t U = Q(U)
$$ 
with initial conditions respectively $f(v) f(w)$ and $g(v)g(w)$.  Therefore in what follows we show that 
$$
 \frac{d}{dt} \mathcal T_\varepsilon(F,G) _{|_{t =0}} \le 0.
$$
As a first step, we rewrite the operator $Q(\cdot)$ defined in (\ref{eq:defn_Q}) as composition of two linear operators. We recall the vector fields $b_k$ 
\begin{equation}\label{eq:defn_b_k}
b_{1}(v-w)=\begin{pmatrix}0\\
w_{3}-v_{3}\\
v_{2}-w_{2}
\end{pmatrix},\quad b_{2}(v-w)=\begin{pmatrix}v_{3}-w_{3}\\
0\\
w_{1}-v_{1}
\end{pmatrix},\quad b_{3}(v-w)=\begin{pmatrix}w_{2}-v_{2}\\
v_{1}-w_{1}\\
0
\end{pmatrix},
\end{equation}
and 
$\tilde{b}_{k}(v-w)=\begin{pmatrix}b_{k}\\-b_{k}\end{pmatrix}$.  The vectors $b_k$  span the subspace $\left\langle v-w\right\rangle ^{\perp}$
and verify 
$$b_{1}\otimes b_{1}+b_{2}\otimes b_{2}+b_{3}\otimes b_{3}=|v-w|^2K(v-w).$$
Using that $\tilde{b}_{k}(v-w)$ is divergence free and perpendicular to $\begin{pmatrix}v-w\\w-v\end{pmatrix}$ it is easy to obtain the expression
\begin{equation}\label{eq:Q_sum_Lb_k}
  Q(F)=\sum_{k=1}^{3}\tilde{b}_{k}\cdot\nabla\left(\tilde{b}_{k}\cdot\nabla F\right)=\sum_{k=1}^{3}\left(L_{\tilde{b}_{k}}\circ L_{\tilde{b}_{k}}\right)F,
\end{equation}
where 
$$
L_{b}(F):=b\cdot\nabla F.
$$
Here the key observation is that the flow of the vector field $\tilde{b}_{k}\left(v-w\right)$ induces an isometry in $\mathbb{R}^6$. That is, $\begin{pmatrix}v(t)\\w(t)\end{pmatrix}$, $\begin{pmatrix}\bar{v}(t)\\\bar{w}(t)\end{pmatrix}$ are solutions of the ODE $\begin{pmatrix}
  \dot{v}\\
  \dot{w}
\end{pmatrix}
=\tilde{b}_{k}\left(v-w\right)$
then:
\begin{align*}
  \frac{1}{2}\frac{d}{dt}
  \left|
    \begin{pmatrix}v\\w\end{pmatrix}
    -
    \begin{pmatrix}\bar{v}\\\bar{w}\end{pmatrix}
    \right|^{2}	=0.
\end{align*}
Consequently, the entropic $2$-Wasserstein distance will remain unchanged if we transport two densities along $\tilde b_k$, as stated in the following lemma:
\begin{lemma}\label{lem:flow_isometries}
  Consider two functions $F_k(v,w,t)$, $G_k(v,w,t)$ satisfying the equations
  \begin{equation}\label{eq:evol_Lk}
    \begin{cases}
      \partial_{t}F_{k}=L_{\tilde{b}_{k}}F_{k}=\tilde{b}_{k}(v-w)\cdot\nabla F_{k}\\F_{k}(v,w,0)=F(v,w)
    \end{cases},\quad
    \begin{cases}
      \partial_{t}G_{k}=L_{\tilde{b}_{k}}G_{k}=\tilde{b}_{k}(v-w)\cdot\nabla G_{k}\\ G_{k}(v,w,0)=G(v,w)
    \end{cases}.
  \end{equation}
  Then $t\mapsto \mathcal T_\varepsilon(F_k,G_k)$ is constant, which implies $\left\langle \mathcal T_\varepsilon'(F_k,G_k), (L_{\tilde{b}_k} F_k,L_{\tilde{b}_k} G_k) \right\rangle =0$.
\end{lemma}
\begin{proof}
  Firstly note that we can write (\ref{eq:evol_Lk}) as continuity equations using that $\tilde b_k$ are divergence free:
  \begin{equation*}
    \partial_t F_k + \textnormal{div}\left(-\tilde b_k F_k\right) = 0,\quad 
    \partial_t G_k + \textnormal{div}\left(-\tilde b_k G_k\right) = 0.
  \end{equation*}
  Then it is known (see \cite[Chap. 4]{santambrogio2015OT}) that $F_k(\cdot, t) = (X_t)_\#F$, $G_k(\cdot ,t)=(X_t)_\#G$ where $X_t(v,w)$ represents the flow of the ODE $\begin{pmatrix}\dot v\\\dot w\end{pmatrix}=-\tilde{b}_{k}(v-w)$. By the previous Remark, we see that $X_t$ is an isometry in $\mathbb{R}^6$, which shows
  \begin{equation*}
    \mathcal{T}_{\varepsilon}\left(F_{k}(\cdot,t),G_{k}(\cdot,t)\right)=\mathcal{T}_{\varepsilon}\left((X_{t})_{\#}F,(X_{t})_{\#}G\right)=\mathcal{T}_{\varepsilon}(F,G).\qedhere
  \end{equation*}
\end{proof}
%

Now we are ready to prove the monotonicity of $\mathcal T_\varepsilon(f,g)$ where $f$, $g$ are solutions of the Landau equation for Maxwell molecules.
\begin{proof}[Proof of Theorem \ref{thm:w2_dissipation}]
  Due to Proposition \ref{prop:equality_W2_production}, it is enough to show that ${\left.\frac{d}{dt}\right|_{t=0}\mathcal T_\varepsilon(F,G)\leq0}$ where $F$, $G$ are solutions of the lifted equations $\partial_{t}F=Q(F)$, $\partial_tG=Q(G)$. We rewrite this as
  \begin{equation*}
    \left\langle \mathcal T_\varepsilon'(F,G), (Q(F),Q(G)) \right\rangle \leq 0.
  \end{equation*}
  Consider for $k=1,2,3$ solutions of
  \begin{equation*}
    \begin{cases}
      \partial_{t}F_{k}(v,w,t)=L_{\tilde{b}_{k}}(F_k)\\
      F_{k}(v,w,0)=F(v,w)
    \end{cases},\quad
    \begin{cases}
      \partial_{t}G_{k}(v,w,t)=L_{\tilde{b}_{k}}(G_k)\\
      G_{k}(v,w,0)=G(v,w)
    \end{cases}.
  \end{equation*}
  Since $\mathcal T_\varepsilon(F_{k},G_k)$ is constant, we deduce that
  \begin{align*}
    0 & =\left.\frac{d^{2}}{dt^{2}}\right|_{t=0}\mathcal T_\varepsilon(F_{k},G_k) \\
    & =\left\langle \mathcal T_\varepsilon'(F,G),\big(\partial_{tt}F_{k}(0),\partial_{tt}G_{k}(0)\big)\right\rangle +
    \left\langle \mathcal T_\varepsilon''(F,G)\big(\partial_{t}F_{k}(0),\partial_{t}G_{k}(0)\big),\big(\partial_{t}F_{k}(0),\partial_{t}G_{k}(0)\big)\right\rangle \\
    & =\left\langle \mathcal T_\varepsilon'(F,G),\big((L_{\tilde{b}_{k}}\circ L_{\tilde{b}_{k}})F,
    (L_{\tilde{b}_{k}}\circ L_{\tilde{b}_{k}})G\big)\right\rangle +
    \left\langle \mathcal T_\varepsilon''(F,G)\big(L_{\tilde{b}_{k}}F,L_{\tilde{b}_{k}}G\big),\big(L_{\tilde{b}_{k}}F,L_{\tilde{b}_{k}}G\big)\right\rangle.
  \end{align*}
  Then, the expression (\ref{eq:Q_sum_Lb_k}) for the operator $Q$ shows
  \begin{align}
    \left\langle \mathcal T_\varepsilon'(F,G), (Q(F),Q(G)) \right\rangle  &
     =\sum_{k=1}^{3}\left\langle \mathcal T_\varepsilon'(F,G),\big((L_{\tilde{b}_{k}}\circ L_{\tilde{b}_{k}})F,(L_{\tilde{b}_{k}}\circ L_{\tilde{b}_{k}})G\big)\right\rangle  \nonumber\\
    & =-\sum_{k=1}^{3}\left\langle \mathcal T_\varepsilon''(F,G)\big(L_{\tilde{b}_{k}}F,L_{\tilde{b}_{k}}G\big),\big(L_{\tilde{b}_{k}}F,L_{\tilde{b}_{k}}G\big)\right\rangle \leq 0, \label{form1}
  \end{align}
  where we have used that $\mathcal T_\varepsilon$ is linearly convex, as evidenced by the dual formulation (\ref{eq:dual_form_bHJB}). 
\end{proof}

\textbf{Remark.}
Inequality (\ref{form1}) would provide an explicit expression of the Wasserstein dissipation if one has an explicit formula for $\mathcal T_\varepsilon''(F,G)$.  For $\varepsilon >0$ finding this explicit expression is still an open problem. For $\varepsilon =0$ the first and second author  have formally proven in  \cite{CD25} that 
\begin{align}
   -\sum_{k=1}^{3} &\left\langle \mathcal T''(F,G)\big(L_{\tilde{b}_{k}}(F),L_{\tilde{b}_{k}}(G)\big),\big(L_{\tilde{b}_{k}}(F),L_{\tilde{b}_{k}}(G)\big)\right\rangle	\nonumber\\
    & =-2\sum_{k=1}^{3}\int\left | D^{2}\Phi(v,w)^{-\frac12}\nabla \xi^f_{k}(v,w)-D^{2}\Psi(\overline{v},\overline{w})^{-\frac12}\nabla \xi^g_{k}(\overline{v},\overline{w})\right |^2  
    \Pi_{*}\left[\begin{pmatrix}v\\w\end{pmatrix},\begin{pmatrix}\overline{v}\\\overline{w}\end{pmatrix}\right],\label{exp_next_paper}
  \end{align}
  where $\xi_k^f$, $\xi_k^g$ solve
  \begin{align}
    -\textnormal{div}\Big(f(v)f(w)D^{2}\Phi(v,w)^{-1}\nabla \xi_{k}^{f}(v,w)\Big) &
    =\tilde{b}_{k}(v-w)\cdot\nabla\left(f(v)f(w)\right), \label{eq:aux_maxewllian}	\\
    -\textnormal{div}\Big(g(v)g(w)D^{2}\Psi(v,w)^{-1}\nabla \xi_{k}^{g}(v,w)\Big)&
    =\tilde{b}_{k}(v-w)\cdot\nabla\left(g(v)g(w)\right), \nonumber
  \end{align}
  and $\Phi$, $\Psi$ are the convex functions mapping $(\nabla\Phi)_\#(f\otimes f) = g\otimes g$, $(\nabla \Psi)_\#(g\otimes g) = f\otimes f$. Formula (\ref{exp_next_paper})  shows that the distance to the Maxwellian $\mathcal T(f,m)$ is strictly decreasing in time unless $f=m$. Indeed, assume $\frac{d}{dt}\mathcal T (f,m) = 0$, then we have that
  $$
  0=\frac{d}{dt} \mathcal T(f,m) = - \frac{1}{2} \sum_{k=1}^3
  \int \left\langle D^2\Phi(v,w)\nabla \xi_k^f, \nabla \xi_k^f  \right\rangle f(v)f(w) \,dv\,dw,
  $$
  where we are using that $\nabla \xi_k^m \equiv 0$. This shows that $\nabla \xi_k^f \equiv 0$ as well. Then, remembering that $b_k(v-w)$, $k=1,2,3$ spans $(v-w)^\perp$, (\ref{eq:aux_maxewllian}) implies that
  $$
  K(v-w)(f(w)\nabla f(v) - f(v)\nabla f(w)) = 0. 
  $$
  The above identity is satisfied if and only if $ f =m$ (among the functions with mass $1$, second moment $3$ and first momentum equal to $0$).


\textbf{Remark.}
The method proposed in this section is not suitable when $\alpha  \not\equiv 1$. As noted  in \cite{GS24},  the flow of the ODE
\begin{equation*}
  \begin{pmatrix}\dot{v}\\
    \dot{w}
  \end{pmatrix}=
 \sqrt{ \alpha\left(|v-w|\right)}\tilde{b}_{k}\left(v-w\right),
\end{equation*}
no longer generates an isometry and Lemma \ref{lem:flow_isometries} does not hold. The next section provides an alternative method that adapts to any $\alpha$.


\section{Proof of Theorem \ref{thm:w2_dissipation}: second method}\label{sec:w2_method2}

In this section we prove Theorem \ref{thm:w2_dissipation} using a different method. During the proof we will introduce tools that will be also useful for the proof of Theorem \ref{thm:w2_estimate_non_Coulomb}; therefore we will keep the term $\alpha$ in the computations. 

{\em{Intuition using the heat equation.}}  Before starting our computation on the Landau equation, we illustrate our intuition by presenting our method applied to the heat equation. The argument is based on the dual formulation of the metric  (see \cite{gigli2020benamou}):
\begin{equation}\label{eq:bhj_formulation}
  \mathcal{T}(f,g)= 
  \sup\Big\{ 
    \int u_1g\,dv-\int u_0f \,dv:\underbrace{\partial_{s}u+\frac{1}{2}\left|\nabla u\right|^{2} =0}_{\text{HJ equation}}
    \Big\}.
\end{equation}
We have the following proposition: 
\begin{proposition*}
  Consider two nonnegative solutions in $\mathbb{R}^d$ of the heat equation $\partial_t f = \Delta f$, $\partial_t g = \Delta g$, both with mass $1$. Then the Wasserstein distance $\mathcal T(f,g)$ is decreasing with dissipation
  \begin{equation*}
    \frac{d}{dt}\mathcal T(f,g) = -\int_0^1\int_{\mathbb{R}^d}\left|D^2u\right|^2\rho(x)\,dx\,ds,
  \end{equation*}
  where $\rho(x,s)$, $0\leq s\leq 1$ is an interpolation between $f$ and $g$ given by the equations
  \begin{equation}\label{eq:CE_heat}
    \partial_{s}\rho+\textnormal{div}\left(\rho\nabla u\right)=0,\qquad\rho_{0}=f,\,\rho_{1}=g,\qquad\partial_{s}u+\frac{1}{2}\left|\nabla u\right|^{2}=0.
  \end{equation}
\end{proposition*}
\begin{remark}
  The previous result is analogous to the estimate of contractivity for finite dimensional gradient flows. Given a smooth function $F:\mathbb{R}^d\to\mathbb{R}$, consider two solutions of the gradient flow $\dot x = -\nabla F(x)$, $\dot y = -\nabla F(y)$. Then the squared Euclidean distance $|x-y|^2$ is decreasing with dissipation
  $$
  \frac{1}{2}\frac{d}{dt}|x-y|^2=- \langle\nabla F(p_0),\dot{p}_0\rangle+\langle\nabla F(p_1),\dot{p}_1\rangle = -\int_0^1 D^2F(p_s)[\dot{p}_s,\dot{p}_s]\,ds,
  $$
  where $p_s = sx + (1-s)y$, $0\leq s\leq 1$ is the constant speed geodesic between $x$ and $y$.
\end{remark}

\begin{proof}
  We use the dual formulation of the Wasserstein distance (\ref{eq:bhj_formulation}) 
  It is known that the optimal $u$ provides an interpolation between $f$ and $g$ given by the equation (\ref{eq:CE_heat}). Therefore, we have
  \begin{align*}
    \frac{d}{dt}\mathcal{T}(f,g) &
    =\int_{\mathbb{R}^{d}}u_{1}\Delta g\,dx-\int_{\mathbb{R}^{d}}u_{0}\Delta f\,dx	\\
    &=\int_{0}^{1}\frac{d}{ds}\int_{\mathbb{R}^{d}}\Delta u\rho\,dx\,ds \\
    & =-\int_{0}^{1}\int_{\mathbb{R}^{d}}\left(\frac{1}{2}\Delta\left|\nabla u\right|^{2}\rho+\Delta u\,\textnormal{div}\left(\rho\nabla u\right)\right)\,dx\, ds\\
    & =-\int_{0}^{1}\int_{\mathbb{R}^{d}}\left(\frac{1}{2}\Delta\left|\nabla u\right|^{2}-\nabla\Delta u\cdot\nabla u\right)\rho\,dx\, ds\\
    & =-\frac{1}{2}\int_{0}^{1}\int_{\mathbb{R}^{d}}\left|D^{2}u\right|^{2}\rho\,dx\,ds,
  \end{align*}
  where we have used Bochner's formula $\frac{1}{2}\Delta\left|\nabla u\right|^{2}=\nabla\Delta u\cdot\nabla u+\left|D^{2}u\right|^{2}$.
\end{proof}

The $2$-Wasserstein contractivity of the heat equation is a classical result. In fact, Jordan, Kinderlehrer and Otto showed in \cite{JKO1998FokkerPlanck} how the heat equation could be understood as a gradient flow for the entropy functional in the $2$-Wasserstein space. Contractivity without a rate formally follows by the displacement convexity of the entropy functional.

Let's now use the same method for the Landau equation. We start by computing  $\frac{d}{dt}\mathcal T(f,g)$ ($\varepsilon =0$) where $f,g$ are solutions of the Landau equation for a general $\alpha$.  We recall that $ \mathcal{T}(f,g)$ is defined in  (\ref{eq:bhj_formulation}). Moreover, the supremum is attained at a function $u$ which provides the constant speed geodesic $\rho_s(v)=\rho(v,s)$, $0\leq s\leq 1$ between $f$ and $g$ through the continuity equation
\begin{equation}\label{eq:CE}
  \partial_{s}\rho+\textnormal{div}\left(\rho\nabla u\right)=0,\quad\rho_{0}=f,\quad\rho_{1}=g.
\end{equation}

\begin{lemma}\label{deriv_Wass2_eps=0}
Let $f$, $g$ be two smooth solutions to (\ref{L}), we have 
\begin{align*}
  \frac{d}{dt}\mathcal T (f,g)  
  = &	-\frac{1}{2}\sum_{k=1}^{3}\int_{0}^{1}\iint\left|\nabla\left(\tilde{b}_{k}\cdot\nabla u\right)\right|^{2}\alpha\rho\,dv\,dw\,ds\\
  & +\frac{1}{2}\sum_{k=1}^{3} \int_0^1\iint\frac{\alpha'}{r}\left(\tilde{b}_{k}\cdot\nabla\left(\tilde{b}_{k}\cdot\nabla u\right)\right)\left(\hat{n}\cdot\nabla u\right)\rho\,dv\,dw\,ds.
  \nonumber
\end{align*}
\end{lemma}

\begin{proof}

Using Proposition \ref{prop:equality_W2_production} and the dual formulation (\ref{eq:bhj_formulation}) we see that 
\begin{equation*}
  \frac{d}{dt}\mathcal{T}(f,g)=\frac{1}{2}\iint u_{1}Q(G)\,dv\,dw-\frac{1}{2}\iint u_{0}Q(F)\,dv\,dw,
\end{equation*}
where $F(v,w)=f(v)f(w)$, $G(v,w)=g(v)g(w)$ and $u(v,w,s)$ is a solution in $\mathbb{R}^6$ of the Hamilton-Jacobi equation from (\ref{eq:bhj_formulation}). We note that the previous expression can be rewritten as
\begin{equation*}
  \frac{d}{dt}\mathcal{T}(f,g)=\frac{1}{2}\int_{0}^{1}\frac{d}{ds}\iint u Q(\rho)\,dv\,dw\,ds,
\end{equation*}
where  $\rho$ is the Wasserstein geodesic interpolating $F$ and $G$.  To lighten notation, we drop the integration with respect to $s$ and focus on rewriting $\frac{d}{ds}\iint uQ(\rho)\,dv\,dw\,ds$. For this purpose, we use the decomposition
\begin{equation*}
  Q(F)=\alpha\sum_{k=1}^{3}\tilde{b}_{k}\cdot\nabla\left(\tilde b_{k}\cdot\nabla F\right),
\end{equation*}
where from now on $\nabla$ and $\textnormal{div}$ denote gradient and divergence with respect to $v,w$. We will repeatedly use that $\alpha \tilde b_k$ is divergence free, which shows that $Q$ is self-adjoint:
\begin{align*}
  \frac{d}{ds}\iint & uQ(\rho)  \,dv\,dw =\frac{d}{ds}\iint Q(u)\rho\,dv\,dw\\
  & =-\sum_{k}\iint\left[\alpha\tilde{b}_{k}\cdot\nabla\left(\tilde{b}_{k}\cdot\nabla u\right)\textnormal{div}\left(\rho\nabla u\right)+\frac{1}{2}\alpha\tilde{b}_{k}\cdot\nabla\left(\tilde{b}_{k}\cdot\nabla\left|\nabla u\right|^{2}\right)\rho\right]\,dv\,dw.
\end{align*}
Introducing $r=\left|v-w\right|$ and $\hat n = \begin{pmatrix}v-w\\w-v\end{pmatrix}$, we see that $\nabla \alpha(r) = \frac{\alpha'(r)}{r}\hat n$ so that
\begin{eqnarray}
  \frac{d}{dt}\mathcal{T}(f,g)&=&\frac{1}{2}\sum_{k}\iint\left[\nabla\left(\tilde{b}_{k}\cdot\nabla\left(\tilde{b}_{k}\cdot\nabla u\right)\right)\cdot\nabla u-\frac{1}{2}\tilde{b}_{k}\cdot\nabla\left(\tilde{b}_{k}\cdot\nabla\left|\nabla u\right|^{2}\right)\right]\alpha\rho\nonumber	\\
  && +\frac{1}{2}\sum_{k}\iint\frac{\alpha'}{r}\left(\tilde{b}_{k}\cdot\nabla\left(\tilde{b}_{k}\cdot\nabla u\right)\right)\left(\hat{n}\cdot\nabla u\right)\rho. \label{eq:aux1}
\end{eqnarray}
Our next goal is to rewrite (\ref{eq:aux1}) as 
%
\begin{eqnarray}
  \frac{d}{dt}\mathcal{T}(f,g)&=&-\frac{1}{2}\sum_{k=1}^{3}\int_{0}^{1}\iint\left|\nabla\left(\tilde{b}_{k}\cdot\nabla u\right)\right|^{2}\alpha\rho\,dv\,dw\,ds\nonumber\\
  & & +\frac{1}{2}\sum_{k=1}^{3} \int_0^1\iint\frac{\alpha'}{r}\left(\tilde{b}_{k}\cdot\nabla\left(\tilde{b}_{k}\cdot\nabla u\right)\right)\left(\hat{n}\cdot\nabla u\right)\rho\,dv\,dw\,ds.\label{eq:final_dissipation}
\end{eqnarray}


The last term in \eqref{eq:aux1} already matches \eqref{eq:final_dissipation}, so we focus on the first term of the right hand side of \eqref{eq:aux1}. We decompose it using the basis vectors 
\begin{equation*}
  \tilde{e}_{i}=\begin{pmatrix}e_{i}\\e_{i}\end{pmatrix}, \quad
  \hat{e}_{i}=\begin{pmatrix}e_{i}\\-e_{i}\end{pmatrix},\quad 1\leq i \leq 3,
\end{equation*}
where $e_i$ form the canonical basis of $\mathbb{R}^3$. More precisely, we use the decompositions
\begin{align*}
  \nabla\left(\tilde{b}_{k}\cdot\nabla\left(\tilde{b}_{k}\cdot\nabla u\right)\right)\cdot\nabla u
  = & \; \frac{1}{2}\sum_{i=1}^{3}\left(\tilde{e}_{i}\cdot\nabla\left(\tilde{b}_{k}\cdot\nabla\left(\tilde{b}_{k}\cdot\nabla u\right)\right)\right)\left(\tilde{e}_{i}\cdot\nabla u\right)	\\
  & +\frac{1}{2}\sum_{i=1}^{3}\left(\hat{e}_{i}\cdot\nabla\left(\tilde{b}_{k}\cdot\nabla\left(\tilde{b}_{k}\cdot\nabla u\right)\right)\right)\left(\hat{e}_{i}\cdot\nabla u\right),
\end{align*}
and
\begin{equation*}
  \tilde{b}_{k}\cdot\nabla\left(\tilde{b}_{k}\cdot\nabla\left|\nabla u\right|^{2}\right)=\frac{1}{2}\sum_{i=1}^{3}\tilde{b}_{k}\cdot\nabla\left(\tilde{b}_{k}\cdot\nabla\left(\hat{e}_{i}\cdot\nabla u\right)^{2}\right)+\frac{1}{2}\sum_{i=1}^{3}\tilde{b}_{k}\cdot\nabla\left(\tilde{b}_{k}\cdot\nabla\left(\hat{e}_{i}\cdot\nabla u\right)^{2}\right).
\end{equation*}
This allows us to rewrite \eqref{eq:aux1} as
\begin{align}
  \label{eq:decomp_I1_I2}
  \frac{d}{dt}\mathcal{T}(f,g)
  =& \; \frac{1}{4}\sum_{i,k}\int_0^1\left[I_{1}(\tilde{e}_{i},k)+I_{2}(\tilde{e}_{i},k)\right]\, ds
  +\frac{1}{4}\sum_{i,k}\int_0^1\left[I_{1}(\hat{e}_{i},k)+I_{2}(\hat{e}_{i},k)\right]\,ds	\\
  & +\frac{1}{2}\sum_{k}\int_0^1\iint\frac{\alpha'}{r}\left(\tilde{b}_{k}\cdot\nabla\left(\tilde{b}_{k}\cdot\nabla u\right)\right)\left(\hat{n}\cdot\nabla u\right)\rho\,ds,
  \nonumber
\end{align}
where we set for a fixed vector $e\in \mathbb{R}^6$ the quantities
\begin{align*}
  I_{1}(e,k) & :=\iint e\cdot\nabla\left(\tilde{b}_{k}\cdot\nabla\left(\tilde{b}_{k}\cdot\nabla u\right)\right)\left(e\cdot\nabla u\right)\alpha \rho,\\
  I_{2}(e,k) & :=-\frac{1}{2}\iint\tilde{b}_{k}\cdot\nabla\left(\tilde{b}_{k}\cdot\nabla\left[(e\cdot\nabla u)^{2}\right]\right)\alpha\rho.
\end{align*}
The advantage of using $\tilde{e}_{i}$, $\hat{e}_{i}$ over the canonical basis is the simplicity of the Lie brackets $[\tilde e_i,\tilde b_k]$, $[\hat e_i,\tilde b_k]$, which we will call commutators from now on. We remind the reader that if $a=a(v)$, $b=b(v)$ are vector fields then Lie bracket $[a,b]$ is another vector field that verifies the identity
\begin{equation*}
  a\cdot\nabla\left(b\cdot\nabla u\right)-b\cdot\nabla\left(a\cdot\nabla u\right)=[a,b]\cdot\nabla u
\end{equation*}
for any smooth function $u$. The components of $[a,b]$ are defined as
\begin{equation*}
  [a,b]_{i}=\sum_{j}\left(a_{j}\partial_{v_{j}}b_{i}-b_{j}\partial_{v_{j}}a_{i}\right)=a\cdot\nabla b_{i}-b\cdot\nabla a_{i}.
\end{equation*}
We calculate, all the possible commutators. We have that $[\tilde e_i,\tilde b_k]=0$ for every $i$ and $k$, and that $[\hat e_i,\tilde b_k]$ satisfies:
  \vspace*{-2mm}
  \begin{table}[h!]
    \centering%
    \begin{tabular}{c|c|c|c} 
      $[\hat{e}_{i},\tilde{b}_{k}]$ & $\tilde{b}_{1}$ & $\tilde{b}_{2}$ & $\tilde{b}_{3}$\tabularnewline
      \hline 
      $\hat{e}_{1}$ & $0$ & $-2\hat{e}_{3}$ & $2\hat{e}_{2}$\tabularnewline
      \hline 
      $\hat{e}_{2}$ & $2\hat{e}_{3}$ & $0$ & $-2\hat{e}_{1}$\tabularnewline
      \hline 
      $\hat{e}_{3}$ & $-2\hat{e}_{2}$ & $2\hat{e}_{1}$ & $0$\tabularnewline
    \end{tabular}
  \end{table}
  \vspace*{-2mm}

Now we rewrite $I_1(e,k)$ by commuting $e$, $\tilde b_k$ and integrating by parts keeping in mind that $\textnormal{div}(\alpha\tilde b_k)=0$:
%
%
\begin{align*}
  I_{1}(e,k) &
  =-\iint\left(e\cdot\nabla\left(\tilde{b}_{k}\cdot\nabla u\right)\right)^{2}\alpha\rho
  -\iint\left(e\cdot\nabla\left(\tilde{b}_{k}\cdot\nabla u\right)\right)\left(\tilde{b}_{k}\cdot\nabla\rho\right)\left(e\cdot\nabla u\right)\alpha\\
  & +\iint\left([e,\tilde{b}_{k}]\cdot\nabla\left(\tilde{b}_{k}\cdot\nabla u\right)\right)\left(e\cdot\nabla u\right)\alpha\rho
  +\iint\left(e\cdot\nabla\left(\tilde{b}_{k}\cdot\nabla u\right)\right)\left([e,\tilde{b}_{k}]\cdot\nabla u\right)\alpha\rho.
\end{align*}
Similarly, $I_2(e,k)$ integrating by parts and commuting $\tilde b_k$, $e$:
\begin{align*}
  I_{2}(e,k) 
  & \!=\!\iint\!\left(e\cdot\nabla\left(\tilde{b}_{k}\cdot\nabla u\right)\right)\left(e\cdot\nabla u\right)\left(\tilde{b}_{k}\cdot\nabla\rho\right)\alpha-\!\iint\!\left([e,\tilde{b}_{k}]\cdot\nabla u\right)\left(e\cdot\nabla u\right)\left(\tilde{b}_{k}\cdot\nabla\rho\right)\alpha.
\end{align*}
Note that the second to last term cancels with one of $I_1(e,k)$, yielding
\begin{align*}
  I_{1}(e,k) & +I_{2}(e,k) 
  =-\iint\left(e\cdot\nabla\left(\tilde{b}_{k}\cdot\nabla u\right)\right)^{2}\alpha\rho
  -\iint\left([e,\tilde{b}_{k}]\cdot\nabla u\right)\left(e\cdot\nabla u\right)\left(\tilde{b}_{k}\cdot\nabla\rho\right)\alpha	\\
  & +\iint\left([e,\tilde{b}_{k}]\cdot\nabla\left(\tilde{b}_{k}\cdot\nabla u\right)\right)\left(e\cdot\nabla u\right)\alpha\rho
  +\iint\left(e\cdot\nabla\left(\tilde{b}_{k}\cdot\nabla u\right)\right)\left([e,\tilde{b}_{k}]\cdot\nabla u\right)\alpha\rho.
\end{align*}
Now, to compute (\ref{eq:decomp_I1_I2}) we must set $e=e_i$ and sum in $i$. Using that $[\tilde e_i,\tilde b_k]=0$ for all $i,k$, we have
\begin{align}
  \nonumber
   \frac{1}{4}\sum_{i}\big[I_{1}(\tilde{e}_{i},k)+I_{2}(\tilde{e}_{i},k) +I_{1} &(\hat{e}_{i},k)+I_{2}(\hat{e}_{i},k)\big] 
    =-\frac{1}{2}\iint\left|\nabla\left(\tilde{b}_{k}\cdot\nabla u\right)\right|^{2}\alpha\rho\\
  &-\frac{1}{4}\sum_{i=1}^{3}\iint\left([\hat{e}_{i},\tilde{b}_{k}]\cdot\nabla u\right)\left(\hat{e}_{i}\cdot\nabla u\right)\left(\tilde{b}_{k}\cdot\nabla\rho\right)\alpha \label{eq:vanish1}\\
  & +\frac{1}{4}\sum_{i=1}^{3}\iint\left([\hat{e}_{i},\tilde{b}_{k}]\cdot\nabla\left(\tilde{b}_{k}\cdot\nabla u\right)\right)\left(\hat{e}_{i}\cdot\nabla u\right)\alpha\rho \label{eq:vanish2}\\
  &+\frac{1}{4}\sum_{i=1}^{3}\iint\left(\hat{e}_{i}\cdot\nabla\left(\tilde{b}_{k}\cdot\nabla u\right)\right)\left([\hat{e}_{i},\tilde{b}_{k}]\cdot\nabla u\right)\alpha\rho \label{eq:vanish3}.
\end{align}
Now observe that \eqref{eq:vanish1} is symmetric in $\hat e_i$, $[\hat e_i,\tilde b_k]$. Since the table of commutators is antisymmetric, we deduce that actually \eqref{eq:vanish1} vanishes. 
%
A similar reasoning shows that \eqref{eq:vanish2} and \eqref{eq:vanish3} vanish. Plugging all this information into (\ref{eq:decomp_I1_I2}) shows that
\begin{align}\label{eq:deriv_w2_final}
  \frac{d}{dt}\mathcal T (f,g) 
  =& 	-\frac{1}{2}\sum_{k=1}^{3}\int_{0}^{1}\iint\left|\nabla\left(\tilde{b}_{k}\cdot\nabla u\right)\right|^{2}\alpha\rho\,dv\,dw\,ds\\
  & +\frac{1}{2}\sum_{k=1}^{3} \int_0^1\iint\frac{\alpha'}{r}\left(\tilde{b}_{k}\cdot\nabla\left(\tilde{b}_{k}\cdot\nabla u\right)\right)\left(\hat{n}\cdot\nabla u\right)\rho\,dv\,dw\,ds.
  \nonumber
\end{align}
This concludes the proof. 
 
\end{proof}

With these preliminary computations for $\varepsilon =0$, we are ready to prove Theorem~\ref{thm:w2_dissipation}: 
\begin{proof} \par
  \noindent \textit{Step 1.} We show that the operator 
  \begin{equation*}
    Q(u):=\left(\textnormal{div}_{v}-\textnormal{div}_{w}\right)\left[\left|v-w\right|^{2}K(v-w)\left(\nabla_{v}u-\nabla_{w}u\right)\right]
  \end{equation*}
  commutes with the Laplacian $\Delta_{v,w}$, namely
  \begin{equation*}
    Q\Delta_{v,w}u=\Delta_{v,w}Qu.
  \end{equation*}
We use the vectors $\tilde{e}_{i}=\begin{pmatrix}e_{i}\\e_{i}\end{pmatrix}$, $\hat{e}_{i}=\begin{pmatrix}e_{i}\\-e_{i}\end{pmatrix}$ and 
  $\tilde{b}_{k}=\begin{pmatrix}b_{k}\\-b_{k}\end{pmatrix}$, noting that
  \begin{equation*}
    \Delta u=\frac{1}{2}\sum_{i}\tilde{e}_{i}\cdot\nabla\left(\tilde{e}_{i}\cdot\nabla u\right)+\frac{1}{2}\sum_{i}\hat{e}_{i}\cdot\nabla\left(\hat{e}_{i}\cdot\nabla u\right)
    ,\quad Qu=\sum_{k}\tilde{b}_{k}\cdot\nabla\left(\tilde{b}_{k}\cdot\nabla u\right).
  \end{equation*}
  We have 
  \begin{align}
    \Delta Qu & =\frac{1}{2}\sum_{i,k}\tilde{e}_{i}\cdot\nabla\left(\tilde{e}_{i}\cdot\nabla\left(\tilde{b}_{k}\cdot\nabla\left(\tilde{b}_{k}\cdot\nabla u\right)\right)\right)
    +\frac{1}{2}\sum_{i,k}\hat{e}_{i}\cdot\nabla\left(\hat{e}_{i}\cdot\nabla\left(\tilde{b}_{k}\cdot\nabla\left(\tilde{b}_{k}\cdot\nabla u\right)\right)\right)\label{eq:Laplacian_Q}\\
    Q\Delta u & =\frac{1}{2}\sum_{i,k}\tilde{b}_{k}\cdot\nabla\left(\tilde{b}_{k}\cdot\nabla\left(\tilde{e}_{i}\cdot\nabla\left(\tilde{e}_{i}\cdot\nabla u\right)\right)\right)
    +\frac{1}{2}\sum_{i,k}\tilde{b}_{k}\cdot\nabla\left(\tilde{b}_{k}\cdot\nabla\left(\hat{e}_{i}\cdot\nabla\left(\hat{e}_{i}\cdot\nabla u\right)\right)\right).\label{eq:Q_Laplacian}
  \end{align}
  Therefore, we may begin by finding an expression for the commutator
  \begin{align*}
    \left[L_{e},L_{b}\right]u :& =L_{e}L_{u}u-L_{b}L_{e}u	\\
    & =e\cdot\nabla\left(e\cdot\nabla\left(b\cdot\nabla\left(b\cdot\nabla u\right)\right)\right)
    -b\cdot\nabla\left(b\cdot\nabla\left(e\cdot\nabla\left(e\cdot\nabla u\right)\right)\right),
  \end{align*}
  where $e$, $b$ are vectors in $\mathbb{R}^{3}$ which may depend on $v$ and $w$ and $L_{e}u:=e\cdot\nabla\left(e\cdot\nabla u\right)$, $L_{b}u:=b\cdot\nabla\left(b\cdot\nabla u\right)$. To do this we commute the order of differentiation to make the sequence $[e,b,b,e]$.
  \begin{align*}
    L_{e}L_{b}u 
    & =e\cdot\nabla\left(b\cdot\nabla\left(b\cdot\nabla\left(e\cdot\nabla u\right)\right)\right)
    +e\cdot\nabla\left([e,b]\cdot\nabla\left(b\cdot\nabla u\right)\right)
    +e\cdot\nabla\left(b\cdot\nabla\left([e,b]\cdot\nabla u\right)\right).
  \end{align*}
  On the other hand, we have
  \begin{align*}
    L_{b}L_{u}u 
    & =e\cdot\nabla\left(b\cdot\nabla\left(b\cdot\nabla\left(e\cdot\nabla u\right)\right)\right)
    -b\cdot\nabla\left([e,b]\cdot\nabla(e\cdot\nabla u)\right)
    -[e,b]\cdot\nabla\left(b\cdot\nabla\left(e\cdot\nabla u\right)\right).
  \end{align*}
  From here we deduce the expression
  \begin{align*}
    \left[L_{e},L_{b}\right]u & =e\cdot\nabla\left(b\cdot\nabla\left([e,b]\cdot\nabla u\right)\right)
    +[e,b]\cdot\nabla\left(b\cdot\nabla\left(e\cdot\nabla u\right)\right)\\
    & +b\cdot\nabla\left([e,b]\cdot\nabla(e\cdot\nabla u)\right)
    +e\cdot\nabla\left([e,b]\cdot\nabla\left(b\cdot\nabla u\right)\right).
  \end{align*}
  To notice cancellations we rewrite the expressions to be symmetric in $e$ and $[e,b]$. Commuting:
  %
  \begin{align*}
    \left[L_{e},L_{b}\right]u & 
    =e\cdot\nabla\left(b\cdot\nabla\left([e,b]\cdot\nabla u\right)\right)
    +[e,b]\cdot\nabla\left(b\cdot\nabla\left(e\cdot\nabla u\right)\right)\\
    & +b\cdot\nabla\left([e,b]\cdot\nabla(e\cdot\nabla u)\right)
    +b\cdot\nabla\left(e\cdot\nabla\left([e,b]\cdot\nabla u\right)\right)\\
    & +e\cdot\nabla\left(\left[[e,b],b\right]\cdot\nabla u\right)
    +\left[e,b\right]\cdot\nabla\left([e,b]\cdot\nabla u\right).
  \end{align*}
  Plugging this into (\ref{eq:Laplacian_Q}, \ref{eq:Q_Laplacian}) and using that $[\tilde{e}_{i},\tilde{b}_{k}]=0$ yields
  \begin{align}
    \Delta Qu-Q\Delta u & 
    =\frac{1}{2}\sum_{i,k}\left[L_{\hat{e}_{i}},L_{\tilde{b}_{k}}\right]u\nonumber \\
    & =\frac{1}{2}\sum_{i,k}\Big(\hat{e}_{i}\cdot\nabla\left(\tilde{b}_{k}\cdot\nabla\left([\hat{e}_{i},\tilde{b}_{k}]\cdot\nabla u\right)\right)
    +[\hat{e}_{i},\tilde{b}_{k}]\cdot\nabla\left(\tilde{b}_{k}\cdot\nabla\left(\hat{e}_{i}\cdot\nabla u\right)\right)\Big)\label{eq:vanish_EOT1}\\
    & +\frac{1}{2}\sum_{i,k}\Big(\tilde{b}_{k}\cdot\nabla\left([\hat{e}_{i},\tilde b_k]\cdot\nabla(\hat{e}_{i}\cdot\nabla u)\right)
    +\tilde{b}_{k}\cdot\nabla\left(\hat{e}_{i}\cdot\nabla\left([\hat{e}_{i},\tilde{b}_{k}]\cdot\nabla u\right)\right)\Big)\label{eq:vanish_EOT2}\\
    & +\frac{1}{2}\sum_{i,k}\Big(\hat{e}_{i}\cdot\nabla\left(\left[[\hat{e}_{i},\tilde{b}_{k}],\tilde{b}_{k}\right]\cdot\nabla u\right)
    +[\hat{e}_{i},\tilde{b}_{k}]\cdot\nabla\left([\hat{e}_{i},\tilde{b}_{k}]\cdot\nabla u\right)\Big).\label{eq:vanish_EOT3}
  \end{align}
  Expressions \eqref{eq:vanish_EOT1} and \eqref{eq:vanish_EOT2} vanish after summing in $i$. The third expression (\ref{eq:vanish_EOT3}) vanishes in a similar manner. For example, observe that for $k=1$ we have
  \begin{equation*}
    \left[[\hat{e}_{2},\tilde{b}_{1}],\tilde{b}_{1}\right]=2[\hat{e}_{2},\tilde{b}_{1}]=
    -4\hat{e}_{2},\quad\left[[\hat{e}_{3},\tilde{b}_{1}],\tilde{b}_{1}\right]=
    -2[\hat{e}_{3},\tilde{b}_{1}]=-4\hat{e}_{3},
  \end{equation*}
  which shows that with $k=1$ fixed
  \begin{equation*}
    (\ref{eq:vanish_EOT3})=
    -2\hat{e}_{2}\cdot\nabla\left(\hat{e}_{2}\cdot\nabla u\right)+2\hat{e}_{2}\cdot\left(\hat{e}_{2}\cdot\nabla u\right)
    -2\hat{e}_{3}\cdot\nabla\left(\hat{e}_{3}\cdot\nabla u\right)+2\hat{e}_{3}\cdot\left(\hat{e}_{3}\cdot\nabla u\right)=0.
  \end{equation*}
  A similar cancellation happens for other values of $k$, finishing the proof of {\em{Step 1}}.
\vspace{0.3cm}

\noindent\textit{Step 2.} Now we proceed with the proof of expression (\ref{eq:w2_dissipation2}) for the dissipation of entropic optimal transport. We use the following dual formulation 
\begin{equation*}
  \mathcal{T}_{\varepsilon}(f,g)=\varepsilon \mathcal H(f)+\sup\left\{ \int u_{1}g-\int u_{0}f:\partial_{s}u+\frac{1}{2}\left|\nabla u\right|^{2}+\frac{\varepsilon}{2}\Delta u=0\right\} ,
\end{equation*}
where $\mathcal H$ denotes the usual entropy. Moreover, the optimal $u$ provides an interpolation between $f$ and $g$ through the equation
\begin{equation}\label{eq:bHJB}
  \partial\rho_{s}+\textnormal{div}\left(\rho\nabla u\right)=\frac{\varepsilon}{2}\Delta\rho,\quad\rho_{0}=f,\quad\rho_{1}=g.
\end{equation}
The interpolant $\rho$ can be written explicitly in terms of the Schrödinger potentials and the heat kernel, although we make no use of this relation here. See \cite{gigli2020benamou} for more details. Now, if $\partial_tf=q(f)$, $\partial_tg=q(g)$ are solutions of the Landau equation for $\alpha \equiv 1$ then, thanks to Proposition \ref{prop:equality_W2_production}, we have
\begin{align*}
    2\frac{d}{dt}\left[\mathcal T_{\varepsilon}(f,g)\right] & =\left\langle \mathcal T_{\varepsilon}'(F,G),\left(Q(F),Q(G)\right)\right\rangle \\
    & =\varepsilon\frac{d}{dt}\mathcal H(F)+\iint u_{1}Q(G)-\iint u_{0}Q(F)\\
    & =\varepsilon\frac{d}{dt}\mathcal H(F)+\iint Q(u_{1})G-\iint Q(u_{0})F.
\end{align*}
Consider the interpolation between $F$ and $G$ of the form (\ref{eq:bHJB}) and denote it $\rho(v,w,s)$, $0\leq s\leq1$. Then we have
\begin{align*}
    2\frac{d}{dt} & \mathcal T_{\varepsilon} (f,g)  
    = -\varepsilon D(F) +\iint Q(u_{1})G
    -\iint Q(u_{0})F\\
    & =\varepsilon\frac{d}{dt}\mathcal H(F)+ \int_{0}^{1}\iint Q(\partial_{s}u)\rho+Q(u)\partial_{s}\rho\,ds\\
    & =\varepsilon\frac{d}{dt}\mathcal H(F)+ \int_{0}^{1}\iint\left[\nabla Q(u)\cdot\nabla u-\frac{1}{2}Q(\left|\nabla u\right|^{2})\right]\rho\,ds
    +\frac{\varepsilon}{2}\int_{0}^{1}\iint\left[\Delta Qu-Q\Delta u\right]\rho\,ds\\
    & =\varepsilon\frac{d}{dt}\mathcal H(F)-\sum_{k}\int_{0}^{1}\iint\left|\nabla\left(\tilde{b}_{k}\cdot\nabla u\right)\right|^{2}\rho\,ds,
\end{align*}
yielding the expression for the dissipation (\ref{eq:w2_dissipation2}) and finishing the proof of Theorem \ref{thm:w2_dissipation}.

\end{proof}

\section{Proof of Theorem \ref{thm:w2_estimate_non_Coulomb}}\label{sec:proof_w2_estimate_non_Coulomb}
We now prove the the contractive estimates for soft potentials using the previously derived expression of $\frac{d}{dt}\mathcal T(f,g)$ in Lemma \ref{deriv_Wass2_eps=0}.

\begin{proof}[Proof of Theorem \ref{thm:w2_estimate_non_Coulomb}]
The starting point is the expression for the dissipation of  $\mathcal{T}$ found in Lemma \ref{deriv_Wass2_eps=0}. 
We first consider the second term and apply Young's inequality. We get 
\begin{align*}
  \frac{1}{2}\iint\frac{\alpha'}{r}\left(\tilde{b}_{k}\cdot\nabla\left(\tilde{b}_{k}\cdot\nabla u\right)\right)\left(\hat{n}\cdot\nabla u\right)\rho\,dv\,dw 
  \le &  \; \frac{1}{4}\iint\left(\tilde{b}_{k}\cdot\nabla\left(\tilde{b}_{k}\cdot\nabla u\right)\right)^{2}\frac{\alpha}{r^{2}}\rho\,dv\,dw	\\
  &+\frac{1}{4}\iint\frac{(\alpha')^2}{\alpha}\left|\hat n \cdot\nabla u\right|^{2}\rho\,dv\,dw.
\end{align*}

To bound the first term, we make use of the following identity (see \cite{GS24}) for functions ${G:\mathbb{R}^6\longrightarrow \mathbb{R}}$:
\begin{equation*}
  \left|\nabla G\right|^{2}=\frac{1}{2}\sum_{i}\left(\tilde{e}_{i}\cdot\nabla G\right)^{2}+\frac{\left(\hat{n}\cdot\nabla G\right)^{2}}{2r^{2}}+\sum_{k}\frac{\left(\tilde{b}_{k}\cdot\nabla G\right)^{2}}{2r^{2}}.
\end{equation*}
In particular, this shows that
\begin{equation*}
  \frac{1}{4}\int_{0}^{1}\iint\left(\tilde{b}_{k}\cdot\nabla\left(\tilde{b}_{k}\cdot\nabla u\right)\right)^{2}\frac{\alpha}{r^{2}}\rho\,dv\,dw\,ds\leq \frac{1}{2}\int_{0}^{1}\iint\left|\nabla\left(\tilde{b}_{k}\cdot\nabla u\right)\right|^{2}\alpha\rho\,dv\,dw\,ds,
\end{equation*}

Plugging both estimates in \eqref{eq:deriv_w2_final} and denoting by
$$K:=\sup_{r>0}\frac{\alpha'(r)^2}{\alpha(r)r^{\gamma-2}}$$
we get 
\begin{equation*}
  \frac{d}{dt}\mathcal{T}(f,g)\leq\frac{3K}{4}\int_{0}^{1}\iint |v-w|^{\gamma-2}\left|\hat n\cdot\nabla u\right|^{2}\rho\,dv\,dw\,ds.
\end{equation*}
It is now  necessary to introduce $\rho(v,w)\equiv\rho(v)\rho(w)$, $u(v,w)\equiv u(v)+u(w)$, which yields the expression
\begin{align*}
  \frac{d}{dt}\mathcal{T}(f,g) & \leq\frac{3K}{4}\int_{0}^{1}\iint\left|v-w\right|^{\gamma-2}\left|(v-w)\cdot\left(\nabla u(v)-\nabla u(w)\right)\right|^{2}\rho(v)\rho(w)\,dv\,dw\,ds	\\
  & \leq\frac{3K}{4}\int_{0}^{1}\iint\left|v-w\right|^{\gamma}\left|\nabla u(v)-\nabla u(w)\right|^{2}\rho(v)\rho(w)\,dv\,dw\,ds.
\end{align*}
We now focus on the case $-3<\gamma<0$, and we simply bound  
$$\left|\nabla u(v)-\nabla u(w)\right|^{2}\leq2\left(\left|\nabla u(v)\right|^{2}+\left|\nabla u(w)\right|^{2}\right)$$
and reverse the role of $v$, $w$ in one of the terms to obtain
\begin{align*}
  \frac{d}{dt}\mathcal{T}(f,g) & \leq3K\int_{0}^{1}\iint\left|v-w\right|^{\gamma}\left|\nabla u(v)\right|^{2}\rho(v)\rho(w)\,dv\,dw\,ds	\\
  & =3K\int_{0}^{1}\int\left|\nabla u(v)\right|^{2}\rho(v)\bigg(\underbrace{\int\rho(w)\left|v-w\right|^{\gamma}\,dw}_{:=(I_{\gamma}\rho)(v)}\bigg)\,dv\\
  & \leq3K\left\Vert I_{\gamma}\rho\right\Vert _{L^{\infty}\left(\mathbb{R}^{3}\times[0,1]\right)}\mathcal{T}(f,g).
\end{align*}
Next, we bound the singular integral $I_{\gamma}\rho$ by splitting the integral in a ball of radius $R$ and its complement:
\begin{align*}
  (I_{\gamma}\rho)(v) & =\int_{B_{R}(v)}\left|v-w\right|^{\gamma}\rho(w)\,dw+\int_{\mathbb{R}^{d}\setminus B_{R}(v)}\left|v-w\right|^{\gamma}\rho(w)\,dw	\\
  & \leq\left(\int_{B_{R}}\left|w\right|^{\gamma p'}\rho(w)\,dw\right)^{\frac{1}{p'}}\left\Vert \rho\right\Vert _{L^{p}(\mathbb{R}^{3})}+R^{\gamma}\\
  & =\left(\frac{4\pi(p-1)}{p(3+\gamma)-3}\right)^{1-\frac{1}{p}}R^{\gamma+\frac{3}{p'}}\left\Vert \rho\right\Vert _{L^{p}(\mathbb{R}^{3})}+R^{\gamma},
\end{align*}
where $p>\frac{3}{3+\gamma}$. By choosing $R=\left\Vert \rho\right\Vert _{L^{p}}^{-\frac{p'}{3}}$ we obtain the following bound:
\begin{equation*}
  \left\Vert I_{\gamma}\rho\right\Vert _{L^{\infty}\left(\mathbb{R}^{3}\times[0,1]\right)}\leq C_{\gamma,p}\left\Vert \rho\right\Vert _{L^{p}(\mathbb{R}^{3})}^{\beta_{p,\gamma}},\quad C_{\gamma,p}=\left(\frac{4\pi(p-1)}{p(3+\gamma)-3}\right)^{1-\frac{1}{p}}+1,\quad\beta_{p,\gamma}=-\frac{\gamma p}{3(p-1)}.
\end{equation*}
Recall that the functional $\rho\mapsto\int\rho^{p}\,dv$ is convex with respect to the $2$-Wasserstein distance, see \cite{mccann1997convexity}. Interpolation yields:
  \begin{align*}
     \left\Vert \rho_{s}\right\Vert _{L^{p}(\mathbb{R}^{3})}=\left(\int\rho_{s}(v)^{p}\,dv\right)^{\frac{1}{p}} &
     \leq\left((1-s)\int f(v)^{p}\,dv+s\int g(v)^{p}\,dv\right)^{\frac{1}{p}}	\\
    & \leq(1-s)^{\frac{1}{p}}\left\Vert f\right\Vert _{L^{p}(\mathbb{R}^{3})}+s^{\frac{1}{p}}\left\Vert g\right\Vert _{L^{p}(\mathbb{R}^{3})}\leq\left\Vert f\right\Vert _{L^{p}(\mathbb{R}^{3})}+\left\Vert g\right\Vert _{L^{p}(\mathbb{R}^{3})}.
  \end{align*}
In summary, we have shown that for $\alpha(r)=r^{\gamma}$, $-3<\gamma <0$, and $p>\frac{3}{3+\gamma}$, there exists $C_{\gamma,p}>0$ and $0<\alpha_{\gamma,p}<1$ such that
\begin{equation*}
  \frac{d}{dt}\mathcal{T}(f_t,g_t) \leq C_{\gamma,p}K \left(\left\Vert f_{t}\right\Vert _{L^{p}(\mathbb{R}^{3})}^{\beta_{\gamma,p}}+\left\Vert g_{t}\right\Vert _{L^{p}(\mathbb{R}^{3})}^{\beta_{\gamma,p}}\right) \mathcal{T}(f_t,g_t).
\end{equation*}
This proves Theorem \ref{thm:w2_estimate_non_Coulomb} for  $-3<\gamma<0$.

Next, we study the Coulomb case $\gamma = -3$, which requires a special adaptation of the previous computation since $|v-w|^{-3}$ is not locally integrable. We need to use the regularity of $u$, which can be expressed through the Hopf-Lax formula as
  \begin{equation}\label{eq:Hopf_lax}
    u_{s}(v)=-\frac{1}{s}\phi_{s}^{*}(v)+\frac{|v|^{2}}{2s},\qquad\phi_{s}(v)=s\phi(v)+(1-s)\frac{|v|^{2}}{2}.
  \end{equation}
  Since $\phi^{*}$ is convex, this shows $D^{2}u_{s}\leq\frac{1}{s}\textnormal{Id}$. On the other hand, it is easy to see that $\tilde{u}(s,v)=-u(1-s,v)$ also satisfies a similar inequality. In fact, we have the bounds
  \begin{equation*}
    -\frac{1}{1-s}\textnormal{Id}\leq D^{2}u_{s}\leq\frac{1}{s}\textnormal{Id},
  \end{equation*}
  which implies that $\left|\nabla u(v)-\nabla u(w)\right|\leq\frac{1}{\min\left(s,1-s\right)}\left|v-w\right|$.
  We consider a parameter $0<\varepsilon<\frac{1}{2}$ to be optimized and note that $\frac{d}{dt}\mathcal{T}(f_t,g_t)$ is bounded by
\begin{equation*}
  \frac{3K}{4}\int_{0}^{1}\left[\min\left(s,1-s\right)\right]^{-2\varepsilon}\iint\left|v-w\right|^{-3+2\varepsilon}\left|\nabla u(v)-\nabla u(w)\right|^{2(1-\varepsilon)}\rho(v)\rho(w)\, dv\,dw\,ds.
\end{equation*}
%
%
%
Using the inequality $\left(\left|\nabla u(v)\right|+\left|\nabla u(w)\right|\right)^{2(1-\varepsilon)}\leq2^{1-2\varepsilon}\left(\left|\nabla u(v)\right|^{2(1-\varepsilon)}+\left|\nabla u(w)\right|^{2(1-\varepsilon)}\right)$ and swapping the roles of $v$, $w$ in the second resulting term, we see that
\begin{align*}
  \frac{d}{dt}\mathcal{T}(f_t,g_t) & \leq3K\int_{0}^{1}\left[\min\left(s,1-s\right)\right]^{-2\varepsilon}\iint\left|v-w\right|^{-3+2\varepsilon}\left|\nabla u(v)\right|^{2(1-\varepsilon)}\rho(v)\rho(w)\,dv\,dw\,ds	\\
  & =3K\int_{0}^{1}\left[\min\left(s,1-s\right)\right]^{-2\varepsilon}\int\left|\nabla u(v)\right|^{2(1-\varepsilon)}\rho(v)(I_{-3+2\varepsilon}\rho)(v)\,dv\,ds,
\end{align*}
where $(I_\gamma\rho)(v)=\int\rho(w)\left|v-w\right|^{\gamma}\,dw$ as before. This yields
\begin{align*}
  \frac{d}{dt}\mathcal{T}(f_t,g_t) &
  \leq3K\left\Vert I_{-3+2\varepsilon}\rho\right\Vert _{L^{\infty}\left(\mathbb{R}^{3}\times[0,1]\right)}\int_{0}^{1}\left[\min\left(s,1-s\right)\right]^{-2\varepsilon}\int\left|\nabla u(v)\right|^{2(1-\varepsilon)}\rho(v)\,dv\,ds \\
  & \leq3K\left\Vert I_{-3+2\varepsilon}\rho\right\Vert _{L^{\infty}\left(\mathbb{R}^{3}\times[0,1]\right)}\int_{0}^{1}\left[\min\left(s,1-s\right)\right]^{-2\varepsilon}\left(\int\left|\nabla u(v)\right|^{2}\rho(v)\,dv\right)^{1-\varepsilon}\,ds.
\end{align*}
Next, we use that $(\rho,u)$ are the constant-speed geodesic between $f$ and $g$ in the $2$-Wasserstein distance. In particular, this means that $\int\left|\nabla u_{s}(v)\right|^{2}\rho_{s}(v)\,dv$ is independent of $s$ and equal to $\mathcal T(f,g)$. Then, using that $\int_{0}^{1}\left[\min\left(s,1-s\right)\right]^{-2\varepsilon}\,ds=\frac{4^{1-\varepsilon}}{1-2\varepsilon}$, we are left with
\begin{equation}\label{eq:bound_with_eps}
  \frac{d}{dt}\mathcal T(f,g)\leq\frac{12K}{1-2\varepsilon}\left\Vert I_{-3+2\varepsilon}\rho\right\Vert _{L^{\infty}}\mathcal T(f,g)^{1-\varepsilon}.
\end{equation}
Finally, we need to optimize in $\varepsilon$. As a first step, we carefully bound the singular integral with respect to $\varepsilon$. Given $R>0$ we have
\begin{align*}
  (I_{-3+2\varepsilon}\rho)(v) & =\int_{B_{R}(v)}\frac{\rho(w)}{\left|v-w\right|^{3-2\varepsilon}}\,dw+\int_{\mathbb{R}^{3}\setminus B_{R}(v)}\frac{\rho(w)}{\left|v-w\right|^{3-2\varepsilon}}\,dw	\\
  &\leq\left\Vert \rho\right\Vert _{L^{\infty}(\mathbb{R}^{3})}\int_{B_{R}}\frac{dw}{\left|w\right|^{3-2\varepsilon}}+R^{-3+2\varepsilon} =\frac{2\pi}{\varepsilon}\left\Vert \rho\right\Vert _{L^{\infty}(\mathbb{R}^{3})}R^{2\varepsilon}+R^{-3+2\varepsilon}.
\end{align*}
Setting $R=\left\Vert \rho\right\Vert _{L^{\infty}(\mathbb{R}^{3})}^{-\frac{1}{3}}\varepsilon^{\frac{1}{3}}$ yields that
\begin{equation}
  \left\Vert I_{-3+2\varepsilon}\rho\right\Vert _{L^{\infty}(\mathbb{R}^{d})}\leq\left(2\pi+1\right)\left\Vert \rho\right\Vert _{L^{\infty}(\mathbb{R}^{3})}^{1-\frac{2\varepsilon}{3}}\varepsilon^{-1+\frac{2\varepsilon}{3}}\leq\left(2\pi+1\right)\max\left(\left\Vert \rho\right\Vert _{L^{\infty}(\mathbb{R}^{3})},1\right)\varepsilon^{-1}. \label{form*}
\end{equation}
If we plug (\ref{form*}) in (\ref{eq:bound_with_eps})  and assume that $0<\varepsilon<\frac{1}{4}$ we obtain
\begin{equation*}
  \frac{d}{dt} \mathcal T(f,g)\leq 24(2\pi+1)K\max\left(\left\Vert f\right\Vert _{L^{\infty}(\mathbb{R}^{3})}+\left\Vert g\right\Vert _{L^{\infty}(\mathbb{R}^{3})},1\right)\varepsilon^{-1}\left[\mathcal T(f,g)\right]^{1-\varepsilon}.
\end{equation*}
We choose $\varepsilon$ that minimizes $\varepsilon^{-1}\mathcal T(f,g)^{1-\varepsilon}$: 
\begin{equation*}
  \varepsilon=
  \begin{cases}
    \frac{-1}{\log \mathcal T(f,g)}, & \textnormal{if } \mathcal T(f,g)\leq e^{-4}\\
    \frac{1}{4}, & \textnormal{if } \mathcal T(f,g)\geq e^{-4}
    \end{cases}.
\end{equation*}
In the case $\mathcal T(f,g)\leq e^{-4}$, the bound is given by
\begin{equation*}
  \frac{d}{dt}\mathcal T(f,g)\leq-24(2\pi+1)K\max\left(\left\Vert f\right\Vert _{L^{\infty}(\mathbb{R}^{3})}+\left\Vert g\right\Vert _{L^{\infty}(\mathbb{R}^{3})},1\right)\mathcal T(f,g)\log \mathcal T(f,g),
\end{equation*}
and the proof is finished by solving the ode. 
Moreover $\mathcal T(f,g)=0$ if $f_{in} = g_{in}$ from Osgood uniqueness criteria.
\end{proof}

\section*{Acknowledgments}
MGD is partially supported by NSF Grant  DMS-2205937. MPG is partially supported by NSF Grant DMS-2206677 and DMS-2511625. MT is supported by NSF grant
DMS-2206187.  The authors would like to thank Nestor Guillen, Jun Kitagawa, and Luis Silvestre for the enlightening discussions.
\bibliographystyle{unsrt}  
\bibliography{references}

\end{document}